\documentclass{amsart}
\pdfoutput=1

\usepackage{tikz}
\usepackage{amsmath}
\usepackage{amssymb}
\usepackage{amsthm}

\title[Neutral-fermionic presentation]{Neutral-fermionic presentation of the $K$-theoretic $Q$-function}
\author{Shinsuke Iwao}
\address{Department of Mathematics, Tokai University, 4-1-1, Kitakaname, Hiratsuka, Kanagawa 259-1292, Japan.}
\email{iwao@tokai.ac.jp}
\date{\today}

\newtheorem{theorem}{Theorem}[section]
\newtheorem{proposition}[theorem]{Proposition}
\newtheorem{lemma}[theorem]{Lemma}

\newtheorem{example}[theorem]{Example}
\newtheorem{remark}[theorem]{Remark}

\newtheorem{corollary}[theorem]{Corollary}

\def\QQ{\mathord{\mathbb{Q}}}

\def\ZZ{\mathord{\mathbb{Z}}}

\def\ket#1{\vert #1 \rangle}
\def\bra#1{\langle #1 \vert}

\def\Pf{\mathop{\mathrm{Pf}}}
\def\zet#1{\left\vert{#1}\right\vert}
\newcommand{\bb}{b^{(\beta)}}
\newcommand{\dbb}{b^{[\beta]}}
\newcommand{\Hb}{\mathcal{H}^{(\beta)}}
\newcommand{\dHb}{\mathcal{H}^{[\beta]}}
\def\bphi{\phi^{(\beta)}} 
\def\dbphi{\phi^{[\beta]}}

\begin{document}

\begin{abstract}
We show a new neutral-fermionic presentation of Ikeda-Naruse's $K$-theoretic $Q$-functions, which represent a Schubert class in the $K$-theory of coherent sheaves on the Lagrangian Grassmannian.
Our presentation provides a simple description and yields a straightforward proof of two types of Pfaffian formulas for them.
We present a dual space of $G\Gamma$, the vector space generated by all $K$-theoretic $Q$-functions, by constructing a non-degenerate bilinear form which is compatible with the neutral fermionic presentation.
We give a new family of dual $K$-theoretic $Q$-functions, their neutral-fermionic presentations, and Pfaffian formulas.


\end{abstract}

\maketitle

\section{Introduction}\label{sec:intro}

Schur's $Q$-polynomials were first introduced by Schur in the paper~\cite{Schur1911} on the projective representations of the symmetric and alternating groups.
They provide many combinatorial results concerning the representation theory of the symmetric group~\cite{STEMBRIDGE198987} and symmetric polynomials~\cite{macdonald1998symmetric}.
They are also essential for Schubert calculus, in which Schur's $Q$-polynomial represents a Schubert class in the cohomology ring of the complex Lagrangian Grassmannian~\cite{jozefiak1991schur,Pragacz1991sqpoly}.

In 2013, Ikeda-Naruse~\cite{IKEDA201322} introduced a $K$-theoretic analog of Schur's $Q$-function, which represents a Schubert class in the $K$-theory of coherent sheaves on the Lagrangian Grassmannian.
(They more generally introduced the $K$-theoretic \textit{factorial $P$- and $Q$-functions} for generalized flags, which represent a Schubert class in the $T$-equivariant $K$-theory~\cite{kostant1990t}.)
There are many known combinatorial and algebraic properties of $K$-theoretic $Q$-functions that generalize these of original $Q$-functions;
there exists a Hall-Littlewood type formula for them~\cite[\S 2.1]{IKEDA201322};
they are expressed as a ratio of Pfaffians (Nimmo-type formula)~\cite[\S 2.3]{IKEDA201322};
they have a combinatorial expression in terms of \textit{excited Young tableaux}~\cite{graham2015excited,ikeda2011bumping}.
(More development and generalizations are found in  ~\cite{hudson2016pfaffian,HUDSON2017115,ikeda2011bumping,kirillov2017construction,nakagawa2017generating}.)

This work aims to present a new algebraic characterization of the $K$-theoretic $Q$-functions in terms of the \textit{boson-fermion correspondence}.
In the previous papers \cite{iwao2020freefermions,iwao2020freefermion}, the author of this paper gave a free-fermionic presentation of \textit{skew stable Grothendieck polynomials} (and their dual polynomials), which are $K$-theoretic versions of Schur functions.
On the analogy of these results, we give a new presentation of $K$-theoretic $Q$-functions by using \textit{``$\beta$-deformed'' neutral fermions}. 
This generalizes the works by Date-Jimbo-Miwa~\cite{date1983method,jimbo1983solitons}, in which they realize Schur's $Q$-functions as a vacuum expectation value of neutral-fermionic operators.

We also discuss duality.
The vector space generated by $K$-theoretic $Q$-functions is characterized by some algebraic condition, which is called the \textit{$K$-theoretic $Q$-cancellation property}~\cite{IKEDA201322}.
Let $G\Gamma$ denote this vector space.
It is not trivial to find a suitable $\beta$-deformed polynomial space that contains some ``good'' dual basis of $K$-theoretic $Q$-functions.
By using our fermionic presentation, however, we can define a new vector space $g\Gamma$ that possesses desirable properties; there exists a non-degenerate bilinear form $G\Gamma\otimes g\Gamma\to \QQ(\beta)$ compatible with neural-fermionic presentations; $g\Gamma$ is characterized by an algebraic condition, which we will call the \textit{dual $K$-theoretic $Q$-cancellation property} (Section \ref{sec:dual1}).


In this paper, we start with a result of Hudson-Ikeda-Matsumura-Naruse~\cite{HUDSON2017115}, which is described as follows.
Let $\beta$ be a parameter.
We consider the binary operators $\oplus$ and $\ominus$ defined by
\[
x\oplus y=x+y+\beta xy,\quad x\ominus y=\frac{x-y}{1+\beta y}.
\]
Set $[[x]]^a=(x\oplus x)x^{a-1}=(2+\beta x)x^a$ for a positive integer $a$, and $[[x]]^0=1$.

Let $\lambda=(\lambda_1>\dots>\lambda_{r}>0)$ be a strict partition of length $r\leq n$. The \textit{$K$-theoretic $Q$-polynomial} $GQ_\lambda(x_1,\dots,x_n)$ is a symmetric polynomial that is expressed as~\cite{IKEDA201322}:
\begin{align*}\label{eq:first_definition}
&GQ_\lambda(x_1,\dots,x_n)\\
&=\frac{1}{(n-r)!}\sum_{w\in S_n}
w\left[
[[x_1]]^{\lambda_1}[[x_2]]^{\lambda_2}\cdots [[x_r]]^{\lambda_r}
\prod_{i=1}^r\prod_{j=i+1}^n\frac{x_i\oplus x_j}{x_i\ominus x_j}
\right],
\end{align*}
where $w\in S_n$ acts on the set of the variables $x_1,\dots,x_n$ as a permutation.
Notice that $GQ_\lambda(x_1,\dots,x_n)$ reduces to Schur's $Q$-polynomial~\cite[\S III]{macdonald1998symmetric} $Q_\lambda(x_1,\dots,x_n)$ when $\beta=0$.

By taking the limit as $n\to \infty$, we obtain the $K$-theoretic $Q$-function $GQ_\lambda(x)=GQ_\lambda(x_1,x_2,\dots)$ in countably many variables~\cite[\S 3]{IKEDA201322}.
In 2017, Hudson-Ikeda-Matsumura-Naruse~\cite{HUDSON2017115} presented the following Pfaffian formula:
\begin{theorem}[{\cite[Theorem 5.21]{HUDSON2017115}}. See also \cite{nakagawa2017generating}]\label{def-prop:pf_GQ}
Let $\lambda$ be a strict partition of $r$, and $r'$ be the minimum even integer that is equal to or greater than $r$.
For a one-row partition $(n)$, we use the abbreviation $GQ_n(x)=GQ_{(n)}(x)$. 
Let $\mathrm{Pf}(c_{i,j})_{1\leq i,j\leq r}$ denote the Pfaffian of an array $(c_{i,j})_{1\leq i,j\leq r}$ $($see \S \ref{sec:prelim}$)$.
Then the $K$-theoretic $Q$-polynomial $GQ_\lambda(x)$ is expressed as:
\[
GQ_\lambda(x)=\Pf\left(
\tau_{i,j}
\right)_{1\leq i<j\leq r'},
\]
where
\[
\tau_{i,j}=
\begin{cases}
\sum_{p\geq 0,\, p+q\geq 0}
f^{i,j}_{p,q}\cdot

GQ_{\lambda_i+p}(x)GQ_{\lambda_j+q}(x), & j\neq r+1,\\
\sum_{p\geq 0}
f^{i,r+1}_{p}
GQ_{\lambda_i+p}(x), & j= r+1.
\end{cases}
\]
Here the coefficients $f_{p,q}^{i,j}\in \QQ[\beta]$ are defined by\footnote{In \cite[\S 5.5]{HUDSON2017115}, the coefficient $f^{i,j}_{p,q}\in \QQ[\beta]$ ($j<r+1$) is given by the formal expansion
\[
\frac{1}{(1+\beta t_i)^{r'-i-1}}
\frac{1}{(1+\beta t_j)^{r'-j}}
\frac{1-\overline{t_i}/\overline{t_j}}{1-t_i/\overline{t_j}}
=
\sum_{p\geq 0,\ p+q\geq 0}f_{p,q}^{i,j} t_i^{p}t_j^{q},\qquad \left(\overline{t}=0\ominus t=\frac{-t}{1+\beta t}\right),
\]
which is equivalent to our description.
See also~\cite[Lemma 5.19]{HUDSON2017115}.
}
\[
\begin{cases}
\displaystyle
\frac{1}{(1+\beta t_i)^{r'-i}}
\frac{1}{(1+\beta t_j)^{r'-j}}
\frac{t_j-t_i}{t_i+t_j+\beta t_it_j}
=
\sum_{p\geq 0,\ p+q\geq 0}f_{p,q}^{i,j} t_i^{p}t_j^{q}, & j\neq r+1,\\
\displaystyle
\frac{1}{(1+\beta t_i)^{r'-i-1}}
=
\sum_{p\geq 0}f_{p}^{i,r+1} t_i^{p}, & j= r+1.
\end{cases}
\]
These formal expansions are interpreted as Laurent series on  $\{\zet{t_{r'}}>\dots>\zet{t_2}>\zet{t_1}\}$.
\end{theorem}

Our project is to reinterpret the Pfaffian formula in Theorem \ref{def-prop:pf_GQ} from the perspective of boson-fermion correspondence.
This formula seems a bit complicated, but the use of fermions makes it simpler.
We would emphasize the fact that the fermionic presentation of $GQ_\lambda(x)$ in the main theorem (Theorem \ref{thm:main1}) looks quite similar to that of the \textit{stable Grothendieck polynomials} in~\cite{iwao2020freefermions,iwao2020freefermion}.
We hope that our result will reveal deeper common structures among various $K$-theoretic symmetric functions.

Here we refer to other works related to this topic.
In the context of geometry, there are various known generalizations of $GQ_\lambda(x)$.
In~\cite{IKEDA201322}, Ikeda-Naruse introduced the $K$-theoretic \textit{factorial} $Q$-functions, representatives of \textit{equivariant} Schubert classes in the $K$-theory, which come back to $GQ_\lambda(x)$ at the non-equivariant limit.
Nakagawa-Naruse~\cite{nakagawa2017generating,nakagawa2018universal} presented the \textit{universal Hall-Littlewood factorial $P$- and $Q$-functions}, which generalize Ikeda-Naruse's result to the \textit{universal cohomology theory} of general flag bundles.
Their works are based on \textit{Gysin maps} for flag bundles in various settings, and there are numerous previous researches on this topic.
Basic ideas can be found in \cite{fulton2006schubert,pragacz1988enumerative}.
For throughout review, we refer to the introduction of \cite{nakagawa2018universal}.


Integrable models also have significant connections with $K$-theoretic symmetric functions.
In the papers~\cite{motegi2013vertex,motegi2014k}, Motegi-Sakai showed a free-fermionic realization of Grothendieck polynomials by constructing wave functions of quantum mechanics such as TASEP (Totally Asymmetric Simple Exclusion Process) and melting crystal.
Interestingly many algebraic formulas concerning Grothendieck polynomials have been derived from arguments on quantum mechanics (see \cite{motegi2020integrability}, for example).
Gorbounov-Korff~\cite{gorbounov2017quantum} have generalized Motegi-Sakai's result for the factorial case.

This paper is composed of two parts:
The first part consists of Sections \ref{sec:neutral_fermions}--\ref{sec:GQ_lambda} and presents a neutral-fermionic presentation of $GQ_\lambda(x)$.
Their Pfaffian formulas are also shown here.
Section \ref{sec:neutral_fermions} contains a definition of the $\beta$-deformed neutral fermion $\bphi_n$.
In Section \ref{sec:vacuum_ex}, we present how to calculate the vacuum expectation values of $\beta$-deformed operators. 
Section \ref{sec:Heisenberg} introduces new operators $\bb$, $e^{\Hb}$, and $e^\Theta$, which we will use in the main theorem.
We also show their commutation relations.
Section \ref{sec:review_Q} contains a review of the vector space $G\Gamma$ of $K$-theoretic $Q$-cancellation property \cite{IKEDA201322} spanned by all $GQ_\lambda(x)$.
Sections \ref{sec:GQ_n}--\ref{sec:GQ_lambda} present the main theorem of this paper.
We show the neutral-fermionic presentation of $GQ_n(x)$ ($n\in \ZZ$) in Section \ref{sec:GQ_n} and that of $GQ_\lambda(x)$ in Section  \ref{sec:GQ_lambda}.

The second part consists of Sections \ref{sec:dual1}--\ref{sec:Pfaffian_formula_dual} and discusses duality.
In Section \ref{sec:dual1}, we introduce the vector space $g\Gamma$ of \textit{dual $K$-theoretic $Q$-cancellation property}.
Every element of $g\Gamma$ is expressed as a vacuum expectation value by using the \textit{dual $\beta$-deformed neutral fermions} $\dbphi_n$, which we introduce in Section \ref{sec:dualbeta}.
We give a neutral-fermionic presentation of the dual $K$-theoretic $Q$-function $gp_\lambda(x)$.
Section \ref{sec:Pfaffian_formula_dual} shows a Pfaffian formula for $o_\lambda(x)$, a basis of $g\Gamma$ used to obtain a simple expression of $gp_\lambda(x)$.

\section{$\beta$-deformed neutral fermions}\label{sec:neutral_fermions}

\subsection{Review of (original) neutral fermions, Fock spaces, and Wick's theorem}\label{sec:prelim}

We briefly review some basic facts about neutral fermions.
For the readers who are interested in this theme, we recommend Baker's paper~\cite{baker1995symmetric} and Jimbo-Miwa's paper~\cite[\S 6]{jimbo1983solitons}.

We write $[A,B]=AB-BA$ and $[A,B]_+=AB+BA$ throughout the paper.
Let $k$ be a field of characteristic $0$.
(We will put $k=\QQ(\beta)$ in the following sections.)
Let $\mathcal{A}$ denote the ring of neutral fermions $\phi_n$ ($n\in \ZZ$) with the anti-commutation relation
\begin{equation}\label{eq:duality_original}
[\phi_m,\phi_n]_+=2(-1)^m\delta_{m+n,0}.
\end{equation}
Notice that $\phi_0^2=1$ and $\phi_n^2=0$ for $n\neq 0$.

Let $\ket{0}$ and $\bra{0}$ be the \textit{vacuum vectors}:
\[
\phi_{-n}\ket{0}=0,\qquad \bra{0}\phi_n =0,\qquad \mbox{for } n>0.
\]
The \textit{Fock space $\mathcal{F}$} is the $k$-vector space generated by the vectors
\begin{equation}\label{eq:F-vectors}
\phi_{n_1}\phi_{n_2}\cdots \phi_{n_r}\ket{0},\qquad 
(r=0,1,2,\dots,\ n_1>n_2>\dots>n_r\geq 0).
\end{equation}
It is known (see \cite[\S 2]{baker1995symmetric} for example) that the vectors \eqref{eq:F-vectors} are linearly independent.
By using the anti-commutation relation \eqref{eq:duality_original} repeatedly, we can define the left action of $\mathcal{A}$ on $\mathcal{F}$.
Thus $\mathcal{F}$ is a left $\mathcal{A}$-module.

We also consider the $k$-space $\mathcal{E}$ generated by the vectors
\begin{equation}\label{eq:E-vectors}
\bra{0}\phi_{m_1}\phi_{m_2}\cdots \phi_{m_r},\qquad 
(r\geq 0,\ 0\geq m_1>m_2>\dots>m_r).
\end{equation}
We can show that the vectors \eqref{eq:E-vectors} are linearly independent and $\mathcal{E}$ is a right $\mathcal{A}$-module.

There exists a $k$-bilinear map
\[
\mathcal{E}\otimes_k\mathcal{F}\to k,\quad 
\bra{w}\otimes \ket{v}\mapsto \langle{w}\vert v\rangle,
\]
called the \textit{vacuum expectation value} \cite[\S 1 and \S 6]{jimbo1983solitons}, which satisfies  (i) $\langle 0\vert 0\rangle=1$, (ii) $(\bra{w}X) \ket{v}=\bra{w} (X\ket{v})$ for any $X\in \mathcal{A}$, and (iii) $\bra{0}\phi_0\ket{0}=0$.
For any $X\in \mathcal{A}$, we write 
$\bra{w}X\ket{v}:=(\langle{w}\vert X)\ket{v}
=\bra{w}(\vert X\ket{v})
$ simply.

For any odd integer $m$, let 
\[
b_m=\frac{1}{4}\sum_{i\in \ZZ} (-1)^i\phi_{-i-m}\phi_{i}
\]
be the \textit{Heisenberg generator}, which defines the linear maps
$
\ket{v}\mapsto b_m\ket{v}
$
and
$
\bra{w}\mapsto \bra{w}b_m
$,
where $\ket{v}\in \mathcal{F}$ and $\bra{w}\in \mathcal{E}$.
We can show that the following commutation relations hold~\cite[\S 6]{jimbo1983solitons}:
\begin{equation}\label{eq:b_and_phi}
[b_m,b_n]=\frac{m}{2}\delta_{m+n,0},\qquad
[b_m,\phi_n]=\phi_{n-m}.
\end{equation}

The Fock space $\mathcal{F}$ admits the graded structure $\mathcal{F}_0
\subset \mathcal{F}_1
\subset \mathcal{F}_2
\subset \mathcal{F}_3\subset\cdots=\mathcal{F}$,
where the subspace $\mathcal{F}_s$ is the $k$-span of the vectors $\phi_{n_1}\phi_{n_2}\cdots \phi_{n_r}\ket{0}$ with
\[
r\geq 0,\quad n_1>n_2>\dots >n_r\geq 0,\quad n_1+n_2+\cdots+n_r=s.
\]
Obviously $\mathcal{F}_s$ is finite dimensional.
(For example, $\mathcal{F}_0=k\cdot \ket{0}\oplus k\cdot \phi_0\ket{0}$.)
We write $\mathcal{F}_s=\{0\}$ when $s$ is negative.
Similarly, $\mathcal{E}$ has the gradation 
$
\mathcal{E}_0\subset \mathcal{E}_{-1}\subset \mathcal{E}_{-2}\subset \cdots=\mathcal{E}$,
where $\mathcal{E}_s$ is the $k$-span of the vectors $\bra{0}\phi_{m_1}\dots \phi_{m_r}$ with
\[
r\geq 0,\quad 0\geq m_1>m_2>\dots>m_r,\quad m_1+m_2+\dots+m_r=s.
\]
From \eqref{eq:duality_original}, we have
\begin{equation}\label{eq:graded_action}
\begin{gathered}
\ket{v}\in \mathcal{F}_s\ \Rightarrow\  
\begin{cases}
\phi_{m}\ket{v}\in \mathcal{F}_{s+m},\\
b_{m}\ket{v}\in \mathcal{F}_{s-m},
\end{cases}\qquad
\bra{w}\in \mathcal{E}_s\ \Rightarrow\  
\begin{cases}
\bra{w}\phi_{m}\in \mathcal{E}_{s+m},\\
\bra{w}b_{m}\in \mathcal{E}_{s-m}.
\end{cases}
\end{gathered}
\end{equation}

The \textit{Pfaffian} of an array $A=(a_{i,j})_{1\leq i<j\leq 2r}$ is defined by the equation
\[
\Pf(A):=\sum_{\sigma}\mathrm{sgn}(\sigma)a_{\sigma(1),\sigma(2)}a_{\sigma(3),\sigma(4)}\cdots a_{\sigma(2r-1),\sigma(2r)},
\]
where $\sigma$ runs on the set of all elements of the symmetric group $S_{2r}$ with
\begin{gather*}
\sigma(1)<\sigma(3)<\dots<\sigma(2r-1),\\ 
\sigma(1)<\sigma(2),\quad \sigma(3)<\sigma(4),\quad \dots,\quad 
\sigma(2r-1)<\sigma(2r).
\end{gather*}

\begin{theorem}[Wick's theorem]\label{thm:Wick}
For any set of integers $\{n_1,\dots,n_{2r}\}$, we have
\[
\bra{0} \phi_{n_1}\cdots\phi_{n_{2r}}\ket{0}=\Pf(\bra{0} \phi_{n_i}\phi_{n_j} \ket{0})_{1\leq i,j\leq 2r}.
\]
\end{theorem}

\subsection{Definition of $\beta$-deformed neutral fermions}\label{sec:phi(beta)}

In the sequel, we will put $k=\QQ(\beta)$ without otherwise stated.
Let $\phi^{(\beta)}_n$ be the \textit{$\beta$-deformed neutral fermion} defined by
\begin{equation}\label{eq:beta_generating_function}
\sum_{n\geq 0}\phi_n^{(\beta)}z^{n}=
\sum_{n\geq 0}\phi_n\left(z+\frac{\beta}{2}\right)^n,\qquad
\sum_{n> 0}\phi^{(\beta)}_{-n}z^{-n}=
\sum_{n> 0}\phi_{-n}
\left(
\frac{z^{-1}}{1+\frac{\beta}{2}z^{-1}}
\right)^n.
\end{equation}
For example, they are written as
\begin{gather*}
\phi^{(\beta)}_{-1}=\phi_{-1},\quad
\phi^{(\beta)}_{-2}=\phi_{-2}-\frac{\beta}{2}\phi_{-1},\quad
\phi^{(\beta)}_{-3}=\phi_{-3}-\beta\phi_{-2}+\frac{\beta^2}{4}\phi_{-1},\\
\phi^{(\beta)}_0=\phi_0+\frac{\beta}{2}\phi_{1}+\frac{\beta^2}{4}\phi_{2}+\cdots,\quad
\phi^{(\beta)}_{1}=\phi_{1}+\beta\phi_{2}+\frac{3\beta^2}{4}\phi_{3}+\cdots.
\end{gather*}

As we can see, $\phi^{(\beta)}_n$ might be an infinite sum, but it defines a $k$-linear map $\mathcal{E}\ni \bra{w}\mapsto \bra{w}\bphi_n \in\mathcal{E}$ unambiguously.
We should note, however, that $\ket{v}\in \mathcal{F}$ does \textit{not} imply $\bphi_n\ket{v}\in \mathcal{F}$ when $n\geq 0$.

The following lemma is an easy consequence of \eqref{eq:graded_action}:
\begin{lemma}\label{lemma:gradeing_of_bphi_on_E}
Let $s,n\in \ZZ$ and $\bra{w}\in \mathcal{E}_s$.
Then we have $\bra{w}\bphi_n\in \mathcal{E}_{s+n}$.
\end{lemma}

\section{Vacuum expectation values}\label{sec:vacuum_ex}

\subsection{Notes on formal power series}

Let $V$ be a $k$-vector space.
Throughout the paper, we will use the notations $V[[z]]$ and $V((z))$ defined as
\begin{align*}
&V[[z]]=\{v_{0}+v_{1}z^{1}+v_2z^2+\cdots\,|\, v_i\in V  \},\\
&V((z))=\{v_{-N}z^{-N}+v_{-N+1}z^{-N+1}+\cdots\,|\,N\geq 0,\ v_i\in V  \}.
\end{align*}
If $V$ is an $R$-algebra with $R$ a ring, $V[[z]]$ and $V((z))$ are also $R$-algebras.

We will write 
\[
V((z_1))((z_2)):=\{V((z_1))\}((z_2)),\quad
V((z_1))[[z_2]]:=\{V((z_1))\}[[z_2]],\quad etc.
\]
Obviously, we have 
$
V[[z]][[w]]=V[[w]][[z]]
$.
However, neither ``$V((z))((w))=V((w))((z))$'' nor ``$V((z))[[w]]=V[[w]]((z))$'' are true in general.
In fact, $1+z^{-1}w+z^{-2}w^2+z^{-3}w^3+\cdots$ is contained in $\QQ((z))((w))$ but is not in $\QQ((w))((z))$.

We often identify rational functions with their Laurent expansion (over some appropriate region) by convention.
For example ``$f(z,w)=\frac{1}{z-w}$ in $\QQ((z))((w))$'' means 
\begin{equation}\label{eq:f_expand}
f(z,w)=z^{-1}+z^{-2}w+z^{-3}w^2+\cdots,
\end{equation}
while ``$g(z,w)=\frac{1}{z-w}$ in $\QQ((w))((z))$'' means 
\begin{equation}\label{eq:g_expand}
g(z,w)=-w^{-1}-w^{-2}z-w^{-3}z^2-\cdots.
\end{equation}
The expression on the right hand side of \eqref{eq:f_expand} is the Laurent expansion over $\{\zet{z}<\zet{w}\}$ and that of \eqref{eq:g_expand} is the Laurent expansion over $\{\zet{w}<\zet{z}\}$.

In general, a formal series\footnote{Formally, a ``formal series'' is an element of $V[[z_1^{\pm 1},\dots,z_r^{\pm 1} ]]$.} 
\[
\displaystyle \sum_{n_1,n_2,\dots,n_r\in \ZZ}v_{n_1,n_2,\dots,n_r}z_1^{n_1}z_2^{n_2}\cdots z_r^{n_r},\qquad (v_{n_1,n_2,\dots,n_r}\in V)
\]
is contained in $V((z_1))((z_2))\dots ((z_r))$ if and only if 
\begin{enumerate}
\item $\displaystyle\sum_{n_1,\dots,n_{r-1}} v_{n_1,\dots,n_{r-1},n_r}z_1^{n_1}\dots z^{n_{r-1}}_{r-1}$\quad $\in V((z_1))\dots ((z_{r-1}))$ and 
\item there exists some $d$ such that  $v_{n_1,\dots,n_r}\neq 0$ $\Rightarrow$ $n_r+d\geq 0$.
\end{enumerate}
In formal language, this condition is equivalent to say 
\begin{gather}\label{eq:formal_condition_inclusion}
\exists d_r
\forall n_r
\exists d_{r-1}
\forall n_{r-1}\cdots
\exists d_{1}
\forall n_{1}
\left(
v_{n_1,\dots,n_r}\neq 0\Rightarrow
\bigwedge_{i=1}^r(d_i+n_i\geq 0)
\right).
\end{gather}

The following lemma provides a handy sufficient condition for a formal series to be contained in $V((z_1))((z_2))\dots ((z_r))$.
\begin{lemma}\label{lemma:creterion_series}
A formal series
\[
\displaystyle \sum_{n_1,n_2,\dots,n_r\in \ZZ}v_{n_1,n_2,\dots,n_r}z_1^{n_1}z_2^{n_2}\cdots z_r^{n_r},\qquad (v_{n_1,n_2,\dots,n_r}\in V)
\]
is contained in $V((z_1))((z_2))\dots ((z_r))$ if there exists some $d$ such that
\[
v_{n_1,n_2,\dots,n_r}\neq 0
\Rightarrow
\begin{cases}
d+n_r\geq 0,\\ 
d+n_r+n_{r-1}\geq 0,\\
d+n_r+n_{r-1}+n_{r-2}\geq 0,\\
\vdots \\
d+n_r+\cdots +n_{1}\geq 0.
\end{cases}
\]
\end{lemma}
\begin{proof}
Put $d_{r}=d$, $d_{r-1}=d+n_r$, $d_{r-2}=d+n_r+n_{r-1}$, \dots, $d_1=d+n_r+\cdots+n_2$.
Then the condition \eqref{eq:formal_condition_inclusion} is satisfied.
\end{proof}

\subsection{Vacuum expectation value of $\bphi_m\bphi_n$}

Here we demonstrate how to calculate the vacuum expectation value $\bra{0} \phi^{(\beta)}_m\phi^{(\beta)}_n\ket{0}$.
Let us consider the formal series
\[
\bphi(z)=\sum_{n\in \ZZ}\bphi_nz^{n}.
\]
By Lemma \ref{lemma:gradeing_of_bphi_on_E}, $\bra{w}\in \mathcal{E}$ implies $\bra{w}\bphi(z)\in \mathcal{E}((z^{-1}))$.
More generally, we have:
\begin{lemma}
Assume $\bra{w}\in \mathcal{E}$.
Then we have
\[
\bra{w}\bphi(z_1)\bphi(z_2)\cdots \bphi(z_r)\qquad
\in \mathcal{E}((z_r^{-1}))\cdots (((z_2^{-1}))((z_1^{-1})).
\]
\end{lemma}
\begin{proof}
Suppose $\bra{w}\in \mathcal{E}_s$ for some $s$.
Consider the expansion 
\begin{equation*}\label{eq:formal_ex}
\bra{w}\bphi(z_1)\bphi(z_2)\cdots \bphi(z_r)
=
\sum_{n_1,\dots,n_r}
\left\{\bra{w}\bphi_{n_1}\bphi_{n_2}\cdots \bphi_{n_r} \right\}
z_1^{n_1}z_2^{n_2}\cdots z_r^{n_r}.
\end{equation*}
By Lemma \ref{lemma:gradeing_of_bphi_on_E}, the contribution of the index $(n_1,n_2,\dots,n_r)$ on the right hand side 
survives only if 
\[
s+n_1\leq 0,\quad
s+n_1+n_2\leq 0, \quad
s+n_1+n_2+n_3\leq 0, \quad \cdots,\quad
s+n_1+\cdots+n_r\leq 0.
\]
Therefore, the conclusion follows from Lemma \ref{lemma:creterion_series}.
\end{proof}

\begin{proposition}\label{prop:bphi_vs_phi}
Let $\bra{w}\in \mathcal{E}$.
Then we have\footnote{In $k((z^{-1}))$, we have
\[
\left(
z+\frac{\beta}{2}
\right)^{-1}
=z^{-1}-\frac{\beta}{2}z^{-2}+\frac{\beta^2}{4}z^{-3}-\cdots.
\]
}
\[
\bra{w}\bphi(z)=\bra{w}\phi(z+\tfrac{\beta}{2})\qquad \mbox{in}\quad \mathcal{E}((z^{-1})).
\]
\end{proposition}
\begin{proof}
This is a direct consequence of \eqref{eq:beta_generating_function}.
\end{proof}

\begin{proposition}\label{prop:phiphi}
We have
\[
\bra{0}
\phi^{(\beta)}(z)\phi^{(\beta)}(w)
\ket{0}
=\frac{z-w}{z+w+\beta}\quad \mbox{in}\quad k((w^{-1}))((z^{-1})).
\]
Here the rational function $\dfrac{z-w}{z+w+\beta}$ is expanded as
\[
\frac{1-wz^{-1}}{1+wz^{-1}+\beta z^{-1}}
=1-(2w+\beta)z^{-1}+(2w^2+3\beta w+\beta^2)z^{-2}-\cdots,
\]
which is the Laurent expansion on the region $\{\zet{\beta}\ll\zet{w}\ll \zet{z}\}$.
\end{proposition}
\begin{proof}
From \eqref{eq:duality_original}, we can directly prove the following equation 
\[
\bra{0} \phi(z)\phi(w)\ket{0}=\frac{z-w}{z+w}=1-2wz^{-1}+2w^2z^{-2}-\cdots\quad
\mbox{in}\ k((w^{-1}))((z^{-1})).
\]
By using this equation and Proposition \ref{prop:bphi_vs_phi}, we have
\begin{align*}
\bra{0} \bphi(z)\bphi(w)\ket{0}
&=
\bra{0} \phi(z+\tfrac{\beta}{2})\phi(w+\tfrac{\beta}{2})\ket{0}\\
&=\frac{(z+\frac{\beta}{2})-(w+\frac{\beta}{2}) }{(z+\frac{\beta}{2})+(w+\frac{\beta}{2})}
=\frac{z-w}{z+w+\beta}.
\end{align*}
\end{proof}

Proposition \ref{prop:phiphi} implies the nontrivial fact that the vacuum expectation value $\bra{0} \bphi_m\bphi_n\ket{0}$ must be of the form $(\mbox{an integer})\times \beta^{-m-n}$.
For example, we have $\bra{0}\bphi_0\bphi_0\ket{0}=1$, $\bra{0}\bphi_{-1}\bphi_0\ket{0}=-\beta$, and $\bra{0}\bphi_{-1}\bphi_1\ket{0}=-2$.

\section{Three $\beta$-deformed operators
$\bb_n$, $e^{\Hb}$, and $e^\Theta$
}\label{sec:Heisenberg}

We introduce three new operators $\bb_n$, $e^{\Hb}$, and $e^\Theta$, which will appear in the statement of the main theorem.
Commutation relations among these operators are essential building blocks for our calculations.

\subsection{$\beta$-deformed Heisenberg generator $\bb_n$}

At the beginning of this section, we consider the commutator $[b_m,\phi^{(\beta)}_n]=b_m\phi^{(\beta)}_n-\phi^{(\beta)}_nb_m$, which defines the $k$-linear map $\mathcal{E}\ni \bra{w}\mapsto \bra{w}[b_m,\phi^{(\beta)}_n]\in \mathcal{E}$.
\begin{proposition}\label{prop:b_and_phi(beta)}
For $\bra{w}\in \mathcal{E}$, we have 
\[
\bra{w}[b_m,\phi^{(\beta)}(z)]=
\left(
z+\frac{\beta}{2}
\right)^{m}\cdot \bra{w}\phi^{(\beta)}(z)\quad \mbox{in}\quad \mathcal{E}((z^{-1})).
\]
\end{proposition}
\begin{proof}
This follows from \eqref{eq:b_and_phi} and Proposition \ref{prop:bphi_vs_phi}.
\end{proof}
Proposition \ref{prop:b_and_phi(beta)} is simply rephrased as
\begin{equation}\label{eq:bphi_z}
[b_m,\phi^{(\beta)}(z)]=
\left(
z+\frac{\beta}{2}
\right)^{m}\phi^{(\beta)}(z),
\end{equation}
over the vector space
$\mathrm{Hom}_k(\mathcal{E},\mathcal{E}((z^{-1})) )$.

Let us now define the \textit{$\beta$-deformed Heisenberg operator} $\bb_m$ ($m\in \ZZ$) by
\begin{equation}\label{eq:def_of_bb}
\bb_m=\left.\frac{(X-\frac{\beta}{2})^m-(-X-\frac{\beta}{2})^m}{2}\right|_{X^k\mapsto b_k},
\end{equation}
in which $\frac{(X-\frac{\beta}{2})^m-(-X-\frac{\beta}{2})^m}{2}$ is expanded as a polynomial in $X$ if $m\geq 0$ and as a Taylor series in $X^{-1}$ if $m<0$.
Below, we list a few examples:
\begin{gather*}
\bb_0=0,\qquad \bb_1=b_1,\qquad\bb_2=-\beta b_1,\qquad \bb_3=b_3+\frac{3\beta^2}{4}b_1,\\
\bb_{-1}=b_{-1}+\frac{\beta^2}{4}b_{-3}+\frac{\beta^4}{16}b_{-5}+\cdots,\qquad
\bb_{-2}=\beta b_{-3}+\frac{\beta^3}{2} b_{-5}+\cdots.
\end{gather*}

\begin{proposition}\label{prop:bb_bphi}
We have
\[
[\bb_m,\bphi(z)]=\frac{z^m-(-z-\beta)^m}{2}\bphi(z)\quad
\mbox{in}\ \mathrm{Hom}_k(\mathcal{E},\mathcal{E}((z^{-1})) ).
\]
\end{proposition}
\begin{proof}
From Proposition \ref{prop:b_and_phi(beta)} and \eqref{eq:def_of_bb}, we have
\[
[\bb_m,\bphi(z)]=\frac{((z+\frac{\beta}{2})-\frac{\beta}{2})^m-(-(z+\frac{\beta}{2})-\beta)^m}{2}\bphi(z),
\]
which gives the proposition.
\end{proof}

\subsection{The operator $e^{\Hb}$}
Let $p_1,p_2,p_3,\dots$ be (commutative) indeterminate.
Set
\[
\mathcal{H}=
2\sum_{n=1,3,5,\dots}\frac{p_n}{n} b_n,\qquad
\mathcal{H}^{(\beta)}=
2\sum_{n=1}^{\infty}\frac{p_n}{n} \bb_n.
\]
Notice that $\mathcal{H}^{(\beta)}$ reduces to $\mathcal{H}$ when $\beta=0$.
Let $\mathcal{F}[[p]]:=\mathcal{F}[[p_1,p_2,p_3,\dots]]$ and $\mathcal{E}[[p]]:=\mathcal{E}[[p_1,p_2,p_3,\dots]]$ be the vector space of infinite series in $p_1,p_2,\dots$ whose coefficients are elements of $\mathcal{F}$ and $\mathcal{E}$, respectively. 
We can check that the formal series $\mathcal{H}^{(\beta)}$ defines the $k$-linear maps
$
\mathcal{F}\ni\ket{v}\mapsto \Hb\ket{v}\in \mathcal{F}[[p]]
$
and
$
\mathcal{E}\ni\bra{w}\mapsto \bra{w}\Hb\in \mathcal{E}[[p]]
$
unambiguously.

We are interested in the exponential map
\[
e^{\Hb}=\mathrm{id}+\Hb+\frac{{\Hb}^2 }{2!}+\frac{{\Hb}^3}{3!}+\cdots,
\]
which also defines the $k$-linear maps
$\ket{v}\mapsto e^{\Hb}\ket{v}$
and
$
\bra{w}\mapsto \bra{w}e^{\Hb}
$.
To demonstrate a direct calculation of the action of $e^{\Hb}$ is not easy in general, but possible only for some easiest cases.
For example, we have $e^{\Hb}\ket{0}=\ket{0}$ and $e^{\Hb}\phi_{1}\ket{0}=(\phi_1+(2p_1-\beta p_2)\phi_0)\ket{0}$.
An effective method to deal with them is the use of commutative relations of formal series as shown in the following lemma:
\begin{lemma}\label{lemma:e^Hphi^(beta)}
In $\mathrm{Map}(\mathcal{E},\mathcal{E}((z^{-1}))[[p]] )$, we have\footnote{Note the difference between $\mathcal{E}((z^{-1}))[[p]]$ and $\mathcal{E}[[p]]((z^{-1}))$.}
\[
e^{\Hb}\phi^{(\beta)}(z)=
\exp\left\{
\sum_{n=1}^\infty
\frac{p_n}{n}
(z^{n}-(-z-\beta)^n)
\right\}
\phi^{(\beta)}(z)e^{\Hb}.
\]
\end{lemma}
\begin{proof}
This follows from
\begin{equation}\label{eq:formula_commutative}
\quad [A,[A,B]]=[B,[A,B]]=0
\quad\Rightarrow\quad
e^{A}e^{B}=e^{[A,B]}e^{B}e^{A}
\end{equation}
and Proposition \ref{prop:bb_bphi}.
\end{proof}

\subsection{The operator $e^{\Theta}$}

We next consider the $k$-linear map $\mathcal{E}\ni\bra{w}\mapsto \bra{w}\Theta\in \mathcal{E}$ defined by
\begin{gather*}
\Theta:=
2\left(\frac{\beta}{2} b_{-1}+\frac{\beta^3}{2^3}\frac{b_{-3}}{3}+
\frac{\beta^5}{2^5}\frac{b_{-5}}{5}+\cdots\right).
\end{gather*}
The exponential map $e^{\Theta}:\mathcal{E}\to\mathcal{E}$ is also well-defined.
The simplest and most important equation about $e^\Theta$ is $\bra{0}e^\Theta=\bra{0}$.

\begin{lemma}\label{lemma:e^He^Theta}
In $\mathrm{Map}(\mathcal{E},\mathcal{E}[[p]] )$, we have
\[
e^{\Hb}e^{\Theta}=\exp\left(-\sum_{n=1}^\infty
\frac{p_n}{n}(-\beta)^n
\right)e^{\Theta}e^{\Hb}.
\]
\end{lemma}
\begin{proof}
For any polynomial $f(X)$ in $X$, let $\mathrm{Coeff}[f(X);X^m]$ denote the coefficient of $X^m$ in $f(X)$.
From \eqref{eq:b_and_phi}, it follows that
\begin{equation}\label{eq:hojyo}
[\bb_m,b_{-n}]=
\frac{n}{4}
\mathord{\mathrm{Coeff}}\left[
\left(X-\frac{\beta}{2}\right)^m-
\left(-X-\frac{\beta}{2}\right)^m;X^n
\right]\qquad \mbox{for}\quad m,n>0.
\end{equation}
Thus we have
\begin{align*}
&[\Hb,\Theta]\\
&=
\sum_{m=1}^\infty
\sum_{n=1,3,5,\dots}\left[
2\frac{p_m}{m}\bb_m,2\left(\frac{\beta}{2}\right)^n\frac{b_{-n}}{n}
\right]\\
&=
\sum_{m=1}^\infty
\sum_{n=1,3,5,\dots}
\frac{p_m}{m}\left(\frac{\beta}{2}\right)^n
\mathord{\mathrm{Coeff}}\left[
\left(X-\frac{\beta}{2}\right)^m-
\left(-X-\frac{\beta}{2}\right)^m;X^n
\right]\quad (\mbox{Eq.\,\eqref{eq:hojyo}} )\\
&=
\sum_{m=1}^\infty
\sum_{n=0}^\infty
\frac{p_m}{m}\left(\frac{\beta}{2}\right)^n
\mathord{\mathrm{Coeff}}\left[
\left(X-\frac{\beta}{2}\right)^m-
\left(-X-\frac{\beta}{2}\right)^m;X^n
\right]\\
&\hspace{15.5em}
(\hbox{\mbox{$(X-\tfrac{\beta}{2})^m-
(-X-\tfrac{\beta}{2})^m$ is an odd function}})\\
&=
\sum_{m=1}^\infty
\frac{p_m}{m}
\left\{
\left(\frac{\beta}{2}-\frac{\beta}{2}\right)^m-
\left(-\frac{\beta}{2}-\frac{\beta}{2}\right)^m
\right\}=-
\sum_{m=1}^\infty
\frac{p_m}{m}(-\beta)^m.
\end{align*}
Therefore, the lemma follows from \eqref{eq:formula_commutative}.
\end{proof}

\begin{proposition}\label{prop:eTheta_action_on_phi(beta)}
We have
\[
e^{\Theta}\phi^{(\beta)}(z)=(
1+\beta z^{-1}
)\cdot \phi^{(\beta)}(z)e^{\Theta}\quad\mbox{in}\quad
\mathrm{Map}(\mathcal{E},\mathcal{E}((z^{-1}))),
\]
where $(1+\beta z^{-1})^{-1}$ is expanded as $(1+\beta z^{-1})^{-1}=1-\beta z^{-1}+\beta^2 z^{-2}-\cdots$.
\end{proposition}
\begin{proof}
Write $\hat{z}=\left(z+\frac{\beta}{2}\right)^{-1}=\frac{z^{-1}}{1+\frac{\beta}{2}z^{-1}}$.
By using the formula
\[
e^AXe^{-A}=X+[A,X]+\frac{1}{2!}[A,[A,X]]+\frac{1}{3!}[A,[A,[A,X]]]+\cdots
\]
and Proposition \ref{prop:b_and_phi(beta)}, 
we obtain
\begin{align*}
e^{\Theta}\phi^{(\beta)}(z)e^{-\Theta}
&=
\exp\left\{
2\left(\frac{\beta}{2} \hat{z}+\frac{\beta^3}{2^3}\frac{\hat{z}^{3}}{3}+
\frac{\beta^5}{2^5}\frac{\hat{z}^{5}}{5}+\cdots\right)
\right\}\phi^{(\beta)}(z)
=
\frac
{1+\frac{\beta}{2}\hat{z}}
{1-\frac{\beta}{2}\hat{z}}
\phi^{(\beta)}(z)\\
&=(1+\beta z^{-1})\phi^{(\beta)}(z).
\end{align*}
Recall the Taylor expansion 
$
\log 
(
\frac{1+t}{1-t}
)=
2
(
t+\frac{t^3}{3}+\frac{t^5}{5}+\cdots
)
$
.
\end{proof}

\begin{corollary}\label{cor:e_Theta_bphi}
We have
\begin{align*}
&e^{\Theta}\phi^{(\beta)}_ne^{-\Theta}=\phi^{(\beta)}_n+\beta \phi^{(\beta)}_{n+1},\\
&e^{-\Theta}\phi^{(\beta)}_ne^{\Theta}=
\phi^{(\beta)}_n-\beta \phi^{(\beta)}_{n+1}+\beta^2 \phi^{(\beta)}_{n+2}-\cdots.
\end{align*}
\end{corollary}

\section{The space of $K$-theoretic $Q$-cancellation property}\label{sec:review_Q}

In this section, we deal with the $K$-theoretic $Q$-cancellation property, which was first introduced by Ikeda-Naruse~\cite{IKEDA201322}.
This condition is considered as a $\beta$-deformed version of the original $Q$-cancellation property (see \eqref{eq:crtGamma}) in the theory of Schur's $Q$-functions.

\subsection{Schur's $Q$-function and bilinear form}\label{sec:Q-bilinear}

We first give a brief review of basic facts about Schur's $Q$-functions according to the textbook~\cite[\S III-8]{macdonald1998symmetric}.

Let $\Lambda_{\QQ}$ be the algebra of symmetric functions in $x_1,x_2,x_3,\dots$ over $\QQ$.
We consider the series of symmetric functions $q_n=q_n(x)$ $(n=1,2,3,\dots)$ defined by
\begin{equation}\label{eq:generating_Q}
\prod_{i=1}^{\infty}\frac{1+x_iz}{1-x_iz}=\sum_{n=0}^{\infty}q_n(x)z^n.
\end{equation}
Let $\Gamma_{\QQ}:=\QQ[q_1,q_2,q_3,\dots]$ be the subalgebra of $\Lambda_{\QQ}$ generated by all $q_n$.

\begin{proposition}[{\cite[\S III-8, (8.5)]{macdonald1998symmetric}}]\label{prop:Q-property}
We have
\[
\Gamma_{\QQ}=\QQ[q_1,q_3,q_5,\dots]=\QQ[p_1,p_3,p_5,\dots],
\]
where $p_i=p_i(x)=x_1^i+x_2^i+\cdots$ is the $i$-th power sum.
The $q_1,q_3,q_5,\dots$ are algebraically independent over $\QQ$.
\end{proposition}

From Proposition \ref{prop:Q-property}, it follows that
\begin{equation}\label{eq:crtGamma}
f(x_1,x_2,\dots)\in \Gamma_{\QQ} \iff 
f(t,-t,x_3,x_4,\dots)\mbox{ does not depend on $t$}
\end{equation}
for any symmetric function $f(x_1,x_2,\dots)\in \Lambda_{\QQ}$.

A partition $\lambda$ is said to be \textit{odd} if every its part is odd.
Let $p_\lambda=p_{\lambda_1}p_{\lambda_2}\cdots$ and $q_\lambda=q_{\lambda_1}q_{\lambda_2}\cdots$.

\begin{lemma}[{\cite[\S III-8, (8.6)]{macdonald1998symmetric}}]\label{lemma:basis_of_Gamma}
Each set listed below forms a $\QQ$-basis of $\Gamma_{\QQ}$.
\begin{enumerate}
\item $q_\lambda$ $(\lambda: \mathrm{strict} )$,
\item $q_\lambda$ $(\lambda: \mathrm{odd})$,
\item $p_\lambda$ $(\lambda: \mathrm{odd})$.
\end{enumerate}
\end{lemma}

We consider the bilinear map $\langle 
\cdot,\cdot
\rangle
:\Gamma_{\QQ}\otimes_{\QQ} \Gamma_{\QQ} \to \QQ$ that is defined by
\begin{equation}
\langle
p_\lambda,p_\mu
\rangle
=2^{-\ell(\lambda)}z_\lambda \delta_{\lambda,\mu},\quad 
\end{equation}
for $\lambda,\mu$ odd, where
$z_\lambda=\prod_{i\geq 1}i^{m_i}\cdot m_i!$ and $m_i=m_i(\lambda)=\sharp\{p\,\vert\, \lambda_p=i\}$.
%
Because
\begin{align*}
\prod_{i,j}\frac{1+x_iy_j}{1-x_iy_j}
&=
\exp\left(
2\sum_{n=1,3,5,\dots}\frac{p_n(x)p_n(y)}{n}
\right)
\\
&=
\prod_{n=1,3,5,\dots}
\sum_{k=0}^\infty
\frac{2^k}{k!n^k}p_n(x)^kp_n(y)^k\\
&=\sum_{\lambda : \mathrm{odd}}2^{\ell(\lambda)}z_\lambda^{-1}
p_\lambda(x)p_\lambda(y),
\end{align*}
we have 
\begin{equation}\label{eq:fundamental_dual}
\langle u_\lambda,v_\mu\rangle=\delta_{\lambda,\mu}
\iff
\prod_{i,j}\frac{1+x_iy_j}{1-x_iy_j}
=\sum_{\lambda}u_\lambda(x)v_\lambda(y)
\end{equation}
for arbitrary $\QQ$-basis $\{u_\lambda\}_\lambda$, $\{v_\lambda\}_\lambda$ of $\Gamma_{\QQ}$.

\begin{theorem}[Schur's $Q$-function]
There exists a unique family of symmetric functions $Q_\lambda(x)$ $(\lambda:\mathrm{strict})$ such that
\begin{enumerate}
\item $Q_\lambda$ forms a $\QQ$-basis of $\Gamma$, 
\item $\langle Q_\lambda,Q_\mu\rangle=2^{\ell(\lambda)} \delta_{\lambda,\mu}$,
\item $Q_\lambda$ is expanded as 
\[
Q_\lambda=q_\lambda+\sum_{\mu\supsetneq \lambda:\mathrm{strict}} a_{\lambda,\mu}q_\mu,\qquad
a_{\lambda,\mu}\in \ZZ.
\]
\end{enumerate}
The symmetric function $Q_\lambda$ is called \textit{Schur's $Q$-function}.
\end{theorem}


It is known~\cite{date1983method,jimbo1983solitons} that, under the identification $p_n\mapsto p_n(x)$, Schur's $Q$-function is realized as a vacuum expectation value of neutral-fermionic operators.
For a strict partition $\lambda=(\lambda_1>\dots>\lambda_r>0)$, we let 
\begin{equation}\label{eq:ket_lambda}
\ket{\lambda}:=
\begin{cases}
\phi_{\lambda_1}\phi_{\lambda_2}\dots \phi_{\lambda_r}\ket{0}, & r\mbox{ : even},\\
\phi_{\lambda_1}\phi_{\lambda_2}\dots \phi_{\lambda_r}\phi_{0}\ket{0}, & r\mbox{ : odd}.
\end{cases}
\end{equation}
Then we have
\begin{equation}\label{eq:classical_result}
Q_{\lambda}(x)=\bra{0}e^{\mathcal{H}}\ket{\lambda}.
\end{equation}

\subsection{$K$-theoretic $Q$-cancellation property}\label{sec:K_Gamma}

Write $\Gamma:=\Gamma_{\QQ}\otimes_{\QQ} k$, where $k=\QQ(\beta)$.
Let $\widehat{\Lambda}$ be the completed ring\footnote{\label{footnote:Lambda}
The $\widehat{\Lambda}$ consists of all Schur series
\[
c_1s_{\lambda_1}(x)+c_2s_{\lambda_2}(x)+c_3s_{\lambda_3}(x)+\cdots,\qquad c_i\in k,
\]
in which $s_\lambda(x)$ is regarded as a monomial of ``degree $\ell(\lambda)$''.
For $f(x)=f(x_1,x_2,\dots)\in \widehat{\Lambda}$, we have
\[
\mathrm{deg}\,f(x)>N\iff f(x_1,\dots,x_N,0,0,\dots)=0.
\]
For example, the infinite series $1+e_1+e_2+\cdots$ is contained in $\widehat{\Lambda}$, while $1+h_1+h_2+\cdots$ is not.
A detailed explanation can be found in~\cite{buch2002littlewood}.
} of symmetric functions over $k$.
This complete ring has a linear topology in which the family
\[
\{U_n\}_{n=1,2,\dots},\qquad
U_n=\{f\,\vert\,f(x_1,\dots,x_n,0,0,\dots)=0\}
\]
forms a basis of open neighborhoods of $0$.

Let $G\Gamma\subset \widehat{\Lambda}$ denote the $k$-subalgebra generated by all symmetric functions satisfying
\begin{equation}\label{eq:KQ_property}
f(t,\overline{t},x_3,x_4,\dots) \mbox{ does not depend on $t$, where } \overline{t}=\frac{-t}{1+\beta t}.
\end{equation}
This is seen as a $\beta$-deformed version of \eqref{eq:crtGamma}.
Eq.\ (\ref{eq:KQ_property}) is called the \textit{$K$-theoretic $Q$-cancellation property}.

Let $\widehat{\Gamma}$ be the $k$-subspace of $\widehat{\Lambda}$ generated by all symmetric functions satisfying the (original) $Q$-cancellation property \eqref{eq:crtGamma}.
Topologically, $\widehat{\Gamma}$ is the closure of $\Gamma\subset \widehat{\Lambda}$.
The following proposition provides a simple necessary and sufficient condition for an element of $\widehat{\Lambda}$ to be contained in $G\Gamma$.

\begin{proposition}\label{prop:Gamma_vs_GGamma}
The substitution map
\begin{equation}\label{eq:substitution}
x_i\mapsto \frac{x_i}{1+\frac{\beta}{2}x_i},\quad i=1,2,3,\dots
\end{equation}
induces an isomorphism $\widehat{\Gamma}\to G\Gamma$.
\end{proposition}
\begin{proof}
Let $y_i=\frac{x_i}{1+\frac{\beta}{2}x_i}$.
For any symmetric function $f\in \widehat{\Lambda}$, set $g(x_1,x_2,\dots):=f(y_1,y_2,\cdots)$.
Then we have
$
g(t,\overline{t},x_3,x_4,\dots)=f(T,-T,y_3,y_4,\dots)
$,
where $T=\frac{t}{1+\frac{\beta}{2}t}$.
This implies $f\in \widehat{\Gamma}\iff g \in G\Gamma$.
\end{proof}

The substitution map \eqref{eq:substitution} is also come from arguments on $\beta$-deformed neutral fermions.
Let 
\[
\Hb_-=2\sum_{n=1}^\infty\frac{p_n(x)}{n}\bb_{-n}
\]
be a negative counterpart of $\Hb$.

\begin{lemma}\label{lemma:Hb_pm}
Under the identification $p_n\mapsto p_n(x)$, we have
\begin{equation}\label{eq:Hb_to_H} 
\Hb
=\sum_{n=1,3,5,\dots}\frac{p_n(\frac{x}{1+\frac{\beta}{2}x})}{n}b_n,
\quad
\Hb_-
=\sum_{n=1,3,5,\dots}\frac{p_n(x+\frac{\beta}{2})-p_n(\frac{\beta}{2})}{n}b_{-n}.
\end{equation}
\end{lemma}
\begin{proof}
From the definition of $\bb_m$ \eqref{eq:def_of_bb}, we have
\[
\bb_m=\sum_{\substack{n\mathrm{:odd}\\n\leq m }}
\left(
-\frac{\beta}{2}
\right)^{m-n}
{m\choose n}b_n,\qquad
\bb_{-m}=\sum_{\substack{n\mathrm{:odd}\\n\geq m }}
\left(
-\frac{\beta}{2}
\right)^{n-m}
{-m\choose n-m}b_{-n}
\]
for $m>0$. 
Here the binomial coefficient ${-a\choose b}$ with $a,b\geq 0$ is given by ${-a\choose b}=(-1)^b{a+b-1\choose b}$.
By using them, we have
\begin{align*}
\Hb=
2\sum_{m=1}^\infty\frac{p_m(x)}{m}\bb_m
&=
2\sum_{m=1}^\infty\frac{p_m(x)}{m}
\sum_{\substack{n\mathrm{:odd}\\n\leq m }}
\left(
-\frac{\beta}{2}
\right)^{m-n}
{m\choose n}b_n\\
&=
2\sum_{n\mathrm{:odd}}
\sum_{m=n}^\infty
\frac{p_m(x)}{m}
\left(
-\frac{\beta}{2}
\right)^{m-n}
{m\choose n}b_n\\
&=
2\sum_{n\mathrm{:odd}}
\sum_{m=n}^\infty
\frac{p_m(x)}{n}
\left(
\frac{\beta}{2}
\right)^{m-n}
{-n\choose m-n}b_n\\
&=
2\sum_{n\mathrm{:odd}}
\frac{p_n(
\frac{x}{1+\frac{\beta}{2}x}
)}{n}b_n
\end{align*}
and
\begin{align*}
\Hb_{-}=
2\sum_{m=1}^\infty\frac{p_m(x)}{m}\bb_{-m}
&=
2\sum_{\substack{n\mathrm{:odd}}}
\sum_{m=1}^n
\frac{p_m(x)}{m}
\left(
-\frac{\beta}{2}
\right)^{n-m}
{-m\choose n-m}b_{-n}\\
&=
2\sum_{n\mathrm{:odd}}
\sum_{m=1}^n
\frac{p_m(x)}{n}
\left(
\frac{\beta}{2}
\right)^{n-m}
{n\choose n-m}b_{-n}\\
&=
2\sum_{n\mathrm{:odd}}
\frac{
p_n(
x+\frac{\beta}{2}
)
-
p_n(
\frac{\beta}{2}
)
}{n}b_{-n}.
\end{align*}
\end{proof}

Let $\mathcal{U}\subset G\Gamma$ be the $k$-vector space spanned by all $p_\lambda(\frac{x}{1+\frac{\beta}{2}x})$ with $\lambda$ odd.
By Lemma \ref{lemma:basis_of_Gamma} and Proposition \ref{prop:Gamma_vs_GGamma}, $\mathcal{U}$ is a dense subspace of $G\Gamma$.
From Lemma \ref{lemma:Hb_pm}, we have the following characterization of $G\Gamma$ in terms of neutral fermions.
\begin{proposition}\label{prop:GGamma_newtral}
Under the identification $p_n\mapsto p_n(x)$, we have
\[
Q_\lambda(\tfrac{x}{1+\frac{\beta}{2}x})=
\bra{0}e^{\Hb}\ket{\lambda}
\]
for any strict partition $\lambda$.
In other words, a family $\{\bra{0}e^{\Hb}\ket{\lambda}\,\vert\, \lambda \mathrm{:strict}\}$ forms a $k$-basis of a dense subspace of $G\Gamma$.
\end{proposition}
\begin{proof}
From Lemma \ref{lemma:Hb_pm}, we have
$
e^{\Hb}\phi_n e^{-\Hb}=\left.\left( e^{\mathcal{H}} \phi_ne^{-\mathcal{H}}\right)\right|_{x\mapsto \frac{x}{1+\frac{\beta}{2}x}}
$,
which implies the proposition.
\end{proof}

\section{Neutral-fermionic presentation of $GQ_n(x)$}\label{sec:GQ_n}

In \cite{HUDSON2017115,nakagawa2017generating}, Hudson-Ikeda-Matsumura-Naruse derived the following equation\footnote{
Equation (\ref{eq:GQ_n}) is obtained from Definition 10.1 in the preprint 
``\textit{Degeneracy Loci Classes in K-theory --- Determinantal and Pfaffian Formula ---} (arXiv:1504.02828v3)'' by Hudson-Ikeda-Matsumura-Naruse:
\[
{}_kG\Theta(x,a;u)=\frac{1}{1+\beta u^{-1}}
\prod_{i=1}^{\infty}\frac{1+(u+\beta)x_i}{1+(u+\beta)\overline{x}_i}
\prod_{i=1}^{k}\{1+(u+\beta)a_i\}
\]
by putting $k=0$. 
However, this equation seems to be omitted from the published version \cite{HUDSON2017115}. 
Equation (\ref{eq:GQ_n}) can be found also in \cite[\S 5.2]{nakagawa2017generating}.
}, in which they gave a generating function of $GQ_n(x)$ corresponding to a one-row partition $(n)$:
\begin{equation}\label{eq:GQ_n}
\sum_{n\in \ZZ} GQ_n(x)z^n=
\frac{1}{1+\beta z^{-1}}\prod_{i=1}^\infty\frac{1+(z+\beta)x_i}{1+(z+\beta)\overline{x}_i}.
\end{equation}
Here $\overline{x}=0\ominus x=\dfrac{-x}{1+\beta x}$, and $(1+\beta z^{-1})^{-1}$ is expanded as $1-\beta z^{-1}+\beta^2z^{-2}-\cdots$.
From \eqref{eq:GQ_n}, we know that the $GQ_n(x)$ is contained in $G\Gamma$.

In this section, we will give a proof of the following neutral-fermionic presentation of $GQ_n(x)$:
\begin{theorem}
\label{thm:GQ_n}
Under the identification $p_n\mapsto p_n(x)$, we have 
\[
GQ_n(x)=\bra{0}e^{\Hb}\bphi_ne^{\Theta}\bphi_0e^{\Theta}\ket{0}.
\]
\end{theorem}

The proof of Theorem \ref{thm:GQ_n} will be given by straightforward calculations based on commutation relations which we have shown in the previous sections.
We first prove the following lemma:
\begin{lemma}\label{lemma:for_theorem_GQ_n}
We have
\begin{align*}
&\bra{0}e^{\Hb}\bphi(z)e^{-\Theta}\bphi_0e^{\Theta}\ket{0}\\
&=(1+\beta z^{-1})
\exp\left\{
\sum_{n=1}^\infty
\frac{p_n}{n}
(z^{n}-(-z-\beta)^n+(-\beta)^n)
\right\}
\end{align*}
in $k((z^{-1}))[[p]]$.
\end{lemma}
\begin{proof}
Let $F(z,w)=\bra{0}e^{\Hb}\bphi(z)e^{-\Theta}\bphi(w)e^{\Theta}\ket{0}$.
From Proposition \ref{prop:eTheta_action_on_phi(beta)}, we have
\[
F(z,w)=(1+\beta w^{-1})^{-1}\cdot \bra{0}e^{\Hb}\bphi(z)\bphi(w)\ket{0}.
\]
By using Proposition \ref{prop:phiphi} and Lemma \ref{lemma:e^Hphi^(beta)}, we obtain
\[
F(z,w)=
\frac{1}{1+\beta w^{-1}}
\eta(z)\eta(w)
\frac{z-w}{z+w+\beta}\quad \mbox{in}\quad
k((w^{-1}))((z^{-1}))[[p]],
\]
where 
$\eta(a)=
\exp\left\{
\sum_{n=1}^\infty
\frac{p_n}{n}
(a^{n}-(-a-\beta)^n)
\right\}
$.
In order to calculate the desired value $\bra{0}e^{\Hb}\bphi(z)e^{-\Theta}\bphi_0e^{\Theta}\ket{0}$, we have to expand $F(z,w)$ in the ring $k((w^{-1}))((z^{-1}))[[p]]$ and take the coefficient of $w^0$.
For this, it is convenient to use contour integrals on the complex plane.
We temporally regard $\beta,w,z,p_1,p_2,\dots$ as complex variables.
The desired expansion of $F(z,w)$ is realized as the Laurent expansion on the domain $\{\zet{\beta}<\zet{w}\}\cap \{\zet{w+\beta}<\zet{z}\}$.
Assume $2\zet{\beta}<\zet{z}$.
Letting $R$ be a positive number satisfying $2\zet{\beta}<R<\zet{z}$ (Figure \ref{fig:one}), we can express the desired value as
\[
I=\frac{1}{2\pi i}\oint_{\zet{w}=R} F(z,w)\frac{dw}{w}.
\]
Since $F(z,w)\frac{dw}{w}$ has one pole at $w=-\beta$ inside the contour, we have
\begin{align*}
I&=
\mathop{\mathrm{Res}}_{w=-\beta}
\left[
F(z,w)\frac{1}{w}
\right]
=
\left.F(z,w)\frac{w+\beta}{w}\right|_{w=-\beta}
=\eta(z)\eta(-\beta)\frac{z+\beta}{z},
\end{align*}
which leads the lemma.
\begin{figure}[htbp]
\begin{center}
\begin{tikzpicture}
\draw (0,0) circle[radius = 1.5];
\draw[->,thick] (0,1.5)--(-0.1,1.5);
\coordinate (O) at (0,0) node at (O) [left] {$0$};
\fill (O) circle (2pt);
\coordinate (B) at (-.5,-.7) node at (B) [below right] {$-\beta$};
\fill (B) circle (2pt);
\coordinate (W) at (30:1.5) node at (W) [right] {$w$};
\fill (W) circle (2pt);
\coordinate (Z) at (5,0.5) node at (Z) [right] {$z$};
\fill (Z) circle (2pt);
\end{tikzpicture}
\end{center}
\caption{An integral path on the $w$-plane.
The radius of the contour is $R$.
The differential $F(z,w)\frac{dw}{w}$ has one pole at $w=-\beta$.}
\label{fig:one}
\end{figure}
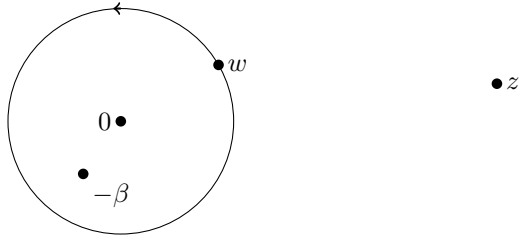
\end{proof}

\begin{proof}[Proof of Theorem \ref{thm:GQ_n}]
Write $\theta(a)=
\exp\left\{
\sum_{n=1}^\infty
\frac{p_n}{n}
a^{n}
\right\}
$.
Notice that, under the identification $p_n\mapsto p_n(x)$, $\theta(a)$ is equal to $\prod_{i=1}^\infty\frac{1}{1-x_ia}$.
Therefore, we have
\begin{align*}
&\bra{0}e^{\Hb}\bphi(z)e^{\Theta}\bphi_0e^{\Theta}\ket{0}\\
&=
\bra{0}(e^{-2\Theta}e^{\Hb}\bphi(z)e^{2\Theta} )(e^{-\Theta}\bphi_0e^{\Theta})\ket{0}
\hspace{3em} (\mbox{By using }\bra{0}e^{-2\Theta}=\bra{0})\\
&=\theta(-\beta)^{-2}(1+\beta z^{-1})^{-2}
\bra{0}e^{\Hb}\bphi(z)e^{-\Theta}\bphi_0e^{\Theta}\ket{0}\\
&\hspace{18em}
(\mbox{By Lemma \ref{lemma:e^He^Theta} and Proposition \ref{prop:eTheta_action_on_phi(beta)}})
\\
&=\theta(-\beta)^{-2}(1+\beta z^{-1})^{-2}\cdot (1+\beta z^{-1})\theta(z)\theta(-z-\beta)^{-1}\theta(-\beta)\\
& \hspace{26em}
(\mbox{By Lemma \ref{lemma:for_theorem_GQ_n}})
\\
&=
\frac{\theta(z)}{(1+\beta z^{-1})\theta(-\beta)\theta(-z-\beta)}
\quad \mbox{in}\quad
k((z^{-1}))[[p]].
\end{align*}
Thus we obtain
\[
\bra{0}e^{\Hb}\bphi(z)e^{\Theta}\bphi_0e^{\Theta}\ket{0}=
\frac{1}{1+\beta z^{-1}}\prod_{i=1}^\infty
\frac{(1+x_i\beta)(1+x_i(z+\beta))}{1-x_iz}.
\]
Comparing this equation to \eqref{eq:GQ_n} leads the theorem.
\end{proof}

\begin{remark}\label{rem:important_remark}
Let 
$
GQ(z)=\sum_{n\in \ZZ} GQ_n(x)z^n
$.
Lemma \ref{lemma:e^Hphi^(beta)} is now rewritten as
\[
e^{\Hb}\bphi(z)e^{-\Hb}=(1+\beta z^{-1})\theta(-\beta)GQ(z) \bphi(z).
\]
\end{remark}

\section{Neutral-fermionic presentation of $GQ_\lambda(x)$}\label{sec:GQ_lambda}

We now proceed to the main theorem, which provides a neutral-fermionic presentation of $GQ_\lambda(x)$ corresponding to a strict partition $\lambda$:
\begin{theorem}\label{thm:main1}
We have
\[
GQ_\lambda(x)=
\begin{cases}
\bra{0}
e^{\Hb}
\phi^{(\beta)}_{\lambda_1}e^{\Theta}
\phi^{(\beta)}_{\lambda_2}e^{\Theta}
\cdots
\phi^{(\beta)}_{\lambda_{r}}e^{\Theta}
\ket{0},
&
\mbox{if $r$ is even},\\
\bra{0}
e^{\Hb}
\phi^{(\beta)}_{\lambda_1}e^{\Theta}
\phi^{(\beta)}_{\lambda_2}e^{\Theta}
\cdots
\phi^{(\beta)}_{\lambda_{r}}e^{\Theta}
\phi^{(\beta)}_{0}e^{\Theta}
\ket{0},
&
\mbox{if $r$ is odd},
\end{cases}
\]
where $GQ_\lambda(x)$ is the $K$-theoretic $Q$-function defined in Theorem \ref{def-prop:pf_GQ}.
\end{theorem}

The proof of Theorem \ref{thm:main1} will be done in \S \ref{sec:proof_main} by comparing the Pfaffian expression in Theorem \ref{def-prop:pf_GQ} and that of the neutral-fermionic expressions.

\subsection{Pfaffian formula I for $GQ_\lambda$}\label{sec:proof_main}

We define the formal series $\Phi=\Phi(z_1,\dots,z_r)$ as
\[
\Phi=
\begin{cases}
\bra{0}
e^{\Hb}\phi^{(\beta)}(z_{1})e^{\Theta}\phi^{(\beta)}(z_2)e^{\Theta}
\cdots
\phi^{(\beta)}(z_{r})e^{\Theta}
\ket{0},& \mbox{if $r$ is even},\\
\bra{0}
e^{\Hb}\phi^{(\beta)}(z_{1})e^{\Theta}\phi^{(\beta)}(z_2)e^{\Theta}
\cdots
\phi^{(\beta)}(z_{r})e^{\Theta}\phi^{(\beta)}_0e^{\Theta}
\ket{0},& \mbox{if $r$ is odd}.
\end{cases}
\]
It is enough to prove that the coefficient of $z_1^{\lambda_1}\cdots z_r^{\lambda_r}$ in $\Phi$ is equal to $GQ_\lambda(x)$.
Let $r'$ be the smallest even number that is equal to or greater than $r$.
By putting 
\begin{align*}
&A_i
=e^{-(r'-i+1)\Theta}\phi^{(\beta)}(z_i)e^{(r'-i+1)\Theta}
\stackrel{ \mathrm{Prop.} \ref{prop:eTheta_action_on_phi(beta)}}{=}(1+\beta z^{-1})^{-(r'-i+1)}\phi^{(\beta)}(z_i)
,\\
&A_{r+1}=e^{-\Theta}\phi^{(\beta)}_0e^{\Theta},
\end{align*}
the generating function $\Phi$ is rewritten as
\begin{align*}
\Phi
&=
\bra{0}
e^{\Hb}e^{r'\Theta}
A_1A_2\cdots A_{r'}
\ket{0}
\stackrel{ \mathrm{Lemma }\, \ref{lemma:e^He^Theta} }{=}
\theta(-\beta)^{-r'}
\bra{0}
e^{\Hb}A_1A_2\cdots A_{r'}
\ket{0}.
\end{align*}
By Wick's theorem (Theorem \ref{thm:Wick}), we have the Pfaffian expression
\begin{align}\label{eq:Phi_matrix}
\Phi=
\theta(-\beta)^{-r'}\Pf
\left(
\bra{0}
e^{\Hb}
A_iA_j
\ket{0}
\right)_{1\leq i<j \leq r'}.
\end{align}
The matrix element $\bra{0}
e^{\Hb}
A_iA_j
\ket{0}$ can be calculated as follows:
\begin{align*}
&\bra{0}
e^{\Hb}
A_iA_j
\ket{0}\\
&=
(1+\beta z_i^{-1})^{-(r'-i+1)}
(1+\beta z_j^{-1})^{-(r'-j+1)}
\bra{0}
e^{\Hb}\bphi(z_i)\bphi(z_j)
\ket{0}
\nonumber\\
&=
\theta(-\beta)^2
(1+\beta z_i^{-1})^{-(r'-i)}
(1+\beta z_j^{-1})^{-(r'-j)}
GQ(z_i)GQ(z_j)
\frac{z_i-z_j}{z_i+z_j+\beta}\\
&\hspace{5em}
(\mbox{By using $e^{\Hb}\ket{0}=\ket{0}$, Proposition \ref{prop:phiphi}, and Remark \ref{rem:important_remark}} )
\end{align*}
for $j< r+1$, and 
\begin{align*}
&\bra{0}
e^{H(t)}
A_iA_{r+1}
\ket{0}\\
&
=
(1+\beta z_i^{-1})^{-(r'-i+1)}
\bra{0}
e^{\Hb}\phi^{(\beta)}(z_i)e^{-\Theta}\bphi_0e^{\Theta}
\ket{0}
\\
&=(1+\beta z_i^{-1})^{-(r'-i+1)}(1+\beta z_i^{-1})^2\theta(-\beta)^2
GQ(z_i)
\qquad
(\mbox{By Lemma \ref{lemma:for_theorem_GQ_n}})
\\
&=\theta(-\beta)^2(1+\beta z_i^{-1})^{-(r'-i-1)}GQ(z_i)
\end{align*}
for $j=r+1$.

Summarizing, we obtain
\begin{lemma}\label{lemma:Phi_Pf}
We have
\begin{equation}\label{eq:Phi(z)}
\Phi=
\Pf\left(
\mathcal{T}_{i,j}
\right)_{1\leq i<j\leq r'},
\end{equation}
where $\mathcal{T}_{i,j}$ is 
\[
\begin{cases}
(1+\beta z_i^{-1})^{-(r'-i)}(1+\beta z_j^{-1})^{-(r'-j)}
GQ(z_i)GQ(z_j)
\dfrac{z_i-z_j}{z_i+z_j+\beta}, & j\neq r+1, \\
(1+\beta z_i^{-1})^{-(r'-i-1)}GQ(z_i), & j=r+1,
\end{cases}
\]
in $G\Gamma((z_{r}^{-1}))\cdots ((z_2^{-1}))((z_1^{-1}))$.
\end{lemma}

\begin{proof}[Proof of Theorem \ref{thm:main1}]
The main theorem \ref{thm:main1} is now obvious.
By putting $t_i=z_i^{-1}$, the Pfaffian formula \eqref{eq:Phi(z)} is rewritten as
\begin{align*}
\Phi
&=
\Pf\left(
\sum_{p\geq 0,\ p+q\geq 0}\sum_{a,b=0}^\infty 
f^{i,j}_{p,q}\cdot
GQ_a(x)
GQ_b(x) 
z_i^{a-p}z_j^{b-q}
\right)_{1\leq i<j\leq r'},
\end{align*}
where $f^{i,j}_{p,q}$ is an element of $\QQ[\beta]$ that has been given in Theorem \ref{def-prop:pf_GQ}.
By comparing the coefficients of $z_1^{\lambda_1}z_2^{\lambda_2}\cdots z_{r}^{\lambda_{r}}$ on the both sides, we conclude that the coefficient of $\Phi$ is equal to $GQ_\lambda(x)$.
This is what we wanted to prove.
\end{proof}


\subsection{Pfaffian formula II for $GQ_\lambda$}\label{sec:anotherP}

We can derive an alternative Pfaffian expression for $GQ_\lambda(x)$ by using the fermionic presentation in Theorem \ref{thm:main1}.
Since the matrix element $\bra{0}
e^{\Hb}
A_iA_j
\ket{0}$ in \eqref{eq:Phi_matrix} can be rewritten as
\begin{align*}
&\bra{0}
e^{\Hb}
A_iA_j
\ket{0}\\
&=
(1+\beta z_i^{-1})^{-(r'-i+1)}
(1+\beta z_j^{-1})^{-(r'-j+1)}
\bra{0}
e^{\Hb}\bphi(z_i)\bphi(z_j)
\ket{0}\\
&=
\theta(-\beta)^{-2}
(1+\beta z_i^{-1})^{-(r'-i-1)}
(1+\beta z_j^{-1})^{-(r'-j)}
\bra{0}
e^{\Hb}\bphi(z_i)e^\Theta\bphi(z_j)e^\Theta
\ket{0}\\
&\hspace{10em}
(\mbox{By using $\bra{0}e^{\Theta}=\bra{0}$, Lemma \ref{lemma:e^He^Theta}, and Proposition \ref{prop:eTheta_action_on_phi(beta)}})
\\
&=
\theta(-\beta)^{-2}
(1+\beta z_i^{-1})^{-(r'-i-1)}
(1+\beta z_j^{-1})^{-(r'-j)}
\Phi(z_i,z_j),
\end{align*}
we have the following new Pfaffian expression
\begin{equation}\label{eq:alterm}
\Phi(z_1,\dots,z_r)
=
\Pf
\left(
\mathcal{R}_{i,j}
\right)_{1\leq i<j \leq r'},
\end{equation}
where
\[
\mathcal{R}_{i,j}=
\begin{cases}
(1+\beta z_i^{-1})^{-(r'-i-1)}
(1+\beta z_j^{-1})^{-(r'-j)}
\Phi(z_i,z_j),& j\neq r+1, \\
(1+\beta z_i^{-1})^{-(r'-i-1)}
\Phi(z_i),& j=r+1.
\end{cases}
\]

Comparing the coefficients of $z_1^{\lambda_1}\cdots z_{r}^{\lambda_{r}}$ of (\ref{eq:alterm}), we obtain
\begin{proposition}[See~{\cite[Theorem 5.7]{nakagawa2017generating}}]\label{prop:another_Pfaffian}
We have
\[
GQ_\lambda(x)
=
\Pf\left(
\rho_{i,j}
\right)_{1\leq i<j \leq r'},
\]
where
\[
\rho_{i,j}=
\begin{cases}
\sum_{k=0}^{\infty}
\sum_{l=0}^{\infty}
\beta^{k+l}
\binom{i+1-r'}{k}
\binom{j-r'}{l}
GQ_{(\lambda_i+k,\lambda_j+l)}(x), & j\neq r+1,\\
\sum_{k=0}^{\infty}
\beta^{k}
\binom{i+1-r'}{k}
GQ_{\lambda_i+k}(x), & j= r+1.
\end{cases}
\]
\end{proposition}

\section{The space of dual $K$-theoretic $Q$-cancellation property}\label{sec:dual1}

\subsection{Definition of $g\Gamma$}

In the latter part of this paper, we deal with a duality.
Let us consider the ``dual space'' of $G\Gamma$ characterized by the following \textit{dual $K$-theoretic $Q$-cancellation property}:
\begin{equation}\label{eq:dual_KQ_property}
f(t,-t-\beta,x_3,x_4,\dots) \mbox{ does not depend on $t$}.
\end{equation}
Let $g\Gamma$ denote the $k$-subspace of $\Lambda$\footnote{Not $\widehat{\Lambda}$.
This means that we need not to work on infinite series of symmetric functions here.} that consists of all symmetric functions satisfying \eqref{eq:dual_KQ_property}.
Obviously, $g\Gamma$ is a ring.


Let $p_n^{(\beta)}(x)=p_n(\frac{x}{1+\frac{\beta}{2}x})$ and $p_n^{[\beta]}(x)=p_n(x+\frac{\beta}{2})-p_n(\frac{\beta}{2})=\sum\limits_{i=1}^{n}{n\choose i}(\frac{\beta}{2})^ip_{i}(x)$.
Both $p_n^{(\beta)}(x)$ and $p_n^{[\beta]}(x)$ come back to $p_n(x)$ when we put $\beta=0$.
For any (not necessarily strict) partition $\lambda=(\lambda_1\geq \dots\geq \lambda_r>0)$, we write
\[
p_\lambda^{(\beta)}:=p_{\lambda_1}^{(\beta)}p_{\lambda_2}^{(\beta)}\dots p_{\lambda_r}^{(\beta)},\qquad
p_\lambda^{[\beta]}:=p_{\lambda_1}^{[\beta]}p_{\lambda_2}^{[\beta]}\dots p_{\lambda_r}^{[\beta]}.
\]

\begin{lemma}
As a $k$-algebra, $g\Gamma$ is generated by $p^{[\beta]}_1,p^{[\beta]}_3,p^{[\beta]}_5,\dots$.
\end{lemma}
\begin{proof}
Let $f(x)$ be an element of $g\Gamma$.
We will prove that $f(x)$ is contained in $k[p^{[\beta]}_1,p^{[\beta]}_3,p^{[\beta]}_5,\dots]$.
Without loss of generality, we assume that $f(x)$ is of the form
\[
f(x)=\sum_{n=0}^N \sum_{\zet{\lambda}\leq M } c_{n,\lambda} \beta^n p_\lambda(x),\qquad c_{n,\lambda}\in \QQ
\]
for some $M,N$.
Let $f_i(x):=\sum_{n+\zet{\lambda}=i}c_{n,\lambda} \beta^n p_\lambda(x)$, then $f(x)$ is decomposed as $f(x)=f_0(x)+f_1(x)+\cdots+f_{M+N}(x)$.
Since each $f_i(t,-t-\beta,x_3,x_4,\dots)$ is a $\QQ$-linear combination of monomials 
\[
\beta^a t^b\cdot p_\lambda(x_3,x_4,\dots)
\quad\mbox{with}\quad 
\zet{\lambda}+a+b=i,
\]
we say that $f(x)\in g\Gamma$ is equivalent to $f_i(x)\in g\Gamma$ for all $i$.
We now will prove $f_i(x)\in g\Gamma$.
The $f_i(x)$ is expanded as
\[
f_i(x)=\beta^d (g_0(x)+\beta g_1(x)+\cdots),\quad g_i(x)\in \Lambda,\quad g_0(x)\neq 0.
\]
The highest term $g_0(x)=\sum_{\zet{\lambda}=i-d}c_{d,\lambda}p_\lambda(x)$ must be contained in $\Gamma_{\QQ}$ because the dual $K$-theoretic $Q$-cancellation property reduces to the original $Q$-cancellation property when $\beta=0$.
Then any partition $\lambda$ with $c_{d,\lambda}\neq 0$ is odd (Lemma \ref{lemma:basis_of_Gamma}).
Let $\lambda_0$ be the maximal $\lambda$ in the lexicographic order such that $c_{d,\lambda}\neq 0$.
The difference $f_i(x)-c_{d,\lambda_0} p_{\lambda_0}^{[\beta]}(x)\in g\Gamma$ is a $k$-linear combination of $p_\lambda(x)$ with $\lambda_0>\lambda$.
Repeating this procedure, we find that $f(x)$ is an element of $k[p^{[\beta]}_1,p^{[\beta]}_3,p^{[\beta]}_5,\dots]$.
\end{proof}

Since $p_1^{[\beta]},p_3^{[\beta]},p_5^{[\beta]},\dots$ are algebraically independent over $k$, the $g\Gamma$ is a $k$-vector space spanned by all $p^{[\beta]}_\lambda$ with $\lambda$ odd.

\subsection{$\beta$-deformed bilinear form}

To develop a duality theory for $K$-theoretic $Q$-functions, we define an appropriate \textit{$\beta$-deformed bilinear form}
\begin{equation}\label{eq:bilinar}
\langle
\cdot,\cdot
\rangle:G\Gamma\otimes_k
g\Gamma\to k,
\end{equation}
which generalizes the bilinear form $\Gamma_{\QQ}\otimes_{\QQ}\Gamma_{\QQ}\to \QQ $ in \S \ref{sec:Q-bilinear}.
For this, we use the following ``Cauchy kernel'' 
\begin{equation}\label{eq:master}
\prod_{i,j}\frac{1-\overline{x_i}y_j}{1-x_iy_j},\quad\mbox{where}\quad \overline{x_i}=\frac{-x_i}{1+\beta x_i}.
\end{equation}
The function \eqref{eq:master} was originally given in Nakagawa-Naruse's work~\cite[Definition 5.3]{nakagawa2016generalized} on geometry of the Grassmann variety of infinite rank.
When we put $\beta=0$, this comes back to $\prod_{i,j}\frac{1+x_iy_j}{1-x_iy_j}$, which was used to define the original bilinear form.
Importantly, \eqref{eq:master} admits the following expansion which beautifully connects the generators of $G\Gamma$ and $g\Gamma$:
\begin{align*}
\prod_{i,j}\frac{1-\overline{x_i}y_j}{1-x_iy_j}
&=
\prod_{i,j}
\left(
\frac{1+\frac{x_i}{1+\frac{\beta}{2}x_i}(y_j+\frac{\beta}{2}) }{1-\frac{x_i}{1+\frac{\beta}{2}x_i}(y_j+\frac{\beta}{2}) }
\cdot
\frac{1-\frac{x_i}{1+\frac{\beta}{2}x_i}\frac{\beta}{2}}
{1+\frac{x_i}{1+\frac{\beta}{2}x_i}\frac{\beta}{2}}
\right)
\\
&=
\exp\left(2\sum_{n=1,3,5,\dots}\frac{p_n(\frac{x}{1+\frac{\beta}{2}x}) 
\{p_n(y+\frac{\beta}{2})-p_n(\frac{\beta}{2})\}
}{n}\right)\\
&=
\prod_{n=1,3,5,\dots}
\sum_{k=0}^\infty
\frac{2^k}{k!n^k}\{p^{(\beta)}_n(x)\}^k\{p^{[\beta]}_n(y)\}^k\\
&=\sum_{\lambda : \mathrm{odd}}2^{\ell(\lambda)}z_\lambda^{-1}
p^{(\beta)}_\lambda(x)p^{[\beta]}_\lambda(y).
\end{align*}
By seeing this, we now define the $\beta$-deformed bilinear form \eqref{eq:bilinar} by
\begin{equation}\label{eq:defi_of_bilin}
\langle
p^{(\beta)}_\lambda,p^{[\beta]}_\mu
\rangle
=2^{-\ell(\lambda)}z_\lambda \delta_{\lambda,\mu}
\end{equation}
for $\lambda,\mu$ odd\footnote{
Rigorously, we first define the $k$-bilinear form $\mathcal{U}\otimes_k g\Gamma\to k$ by \eqref{eq:defi_of_bilin} and then extend to $G\Gamma\otimes_k g\Gamma$ continuously.
This method works because the map $\mathcal{U}\to k$; $f\mapsto \langle f,g\rangle$ is continuous on $\mathcal{U}$ for any $g\in g\Gamma$.
Here $k$ is equipped with the discrete topology.
}.


The master function \eqref{eq:master} is naturally derived from the fermionic setting.
In fact, from Lemma \ref{lemma:Hb_pm}, we obtain
\[
[\Hb(x),\Hb_{-}(y)]=
2\sum_{n=1,3,5,\dots}\frac{p_n(\frac{x}{1+\frac{\beta}{2}x}) 
\{p_n(y+\frac{\beta}{2})-p_n(\frac{\beta}{2})\}
}{n},
\]
which implies
\begin{equation}\label{eq:neutralfemion_vs_symmetricpolynomial}
\bra{0}e^{\Hb(x)}e^{\Hb_{-}(y)}\ket{0}
=\prod_{i,j}
\frac{1-\overline{x_i}y_j}{1-x_iy_j}.
\end{equation}
To relate this equation to the duality theory, we need the anti-isomorphism 
\begin{equation}\label{eq:anti}
\mathcal{A}\to \mathcal{A};\quad X\mapsto X^\ast
\end{equation}
such that $(\phi_n)^\ast=(-1)^n\phi_{-n}$ and $(XY)^\ast=Y^\ast X^\ast$.
This isomorphism is well-defined by virtue of the anti-commutation relation \eqref{eq:duality_original}.
By letting $\bra{0}^\ast =\ket{0}$ and $\ket{0}^\ast=\bra{0}$, the anti-isomorphism \eqref{eq:anti} induces the involution
\begin{equation}\label{eq:anti2}
\mathcal{E}\leftrightarrow \mathcal{F};\quad \bra{0}X\leftrightarrow X^\ast \ket{0}.
\end{equation}
\begin{lemma}\label{lemma:basic_duality}
Write $\bra{\lambda}:=\ket{\lambda}^\ast$ for a strict partition $\lambda$.
Then we have
$
\langle
\lambda |
\mu
\rangle
=2^{\ell(\lambda)}\delta_{\lambda,\mu}
$.
\end{lemma}
\begin{proof}
This lemma can be proved directly by Wick's theorem \ref{thm:Wick}.
\end{proof}

Let $\dbb_n:=(\bb_{-n})^\ast$ and 
\[
\dHb:=(\Hb_{-})^\ast=2\sum_{n=1,3,5,\dots}\frac{p_n}{n}\dbb_{n},\qquad
\dHb_{-}:=(\Hb)^\ast=2\sum_{n=1,3,5,\dots}\frac{p_n}{n}\dbb_{-n}.
\]
Since $\mathcal{E}$ is spanned by the vectors of the form $\bra{\lambda}$ with $\lambda$ strict, we have the following expansions
\[
\bra{0}e^{\Hb(x)}=\sum_{\lambda:\mathrm{strict}} \bra{\lambda}u_\lambda(x),\qquad
\bra{0}e^{\dHb(x)}=\sum_{\lambda:\mathrm{strict}} \bra{\lambda}v_\lambda(x)
\]
for some symmetric functions $u_\lambda$ and $v_\lambda$.
By using Lemma \ref{lemma:basic_duality} and \eqref{eq:neutralfemion_vs_symmetricpolynomial}, we obtain the equation
\[
\sum_{\lambda:\mathrm{strict}} 
2^{\ell(\lambda)}u_\lambda(x)v_\lambda(y)=
\prod_{i,j}\frac{1-\overline{x_i}y_j}{1-x_iy_j},
\]
which implies
\[
\langle
u_\lambda,v_\mu
\rangle
=2^{-\ell(\lambda)}\delta_{\lambda,\mu}.
\]

\begin{proposition}\label{prop:bilin_and_expect}
Suppose $f=\bra{0}e^{\Hb(x)}\ket{v}$ and $g=\bra{0}e^{\dHb(x)}\ket{w}$.
Then we have
\[
\langle
f,g
\rangle
=
\langle
w|v
\rangle.
\]
\end{proposition}
\begin{proof}
It suffices to prove the case when $\ket{v}=\ket{\lambda}$ and $\ket{w}=\ket{\mu}$ for strict partitions $\lambda$, $\mu$.
In this case, we have $f=2^{\ell(\lambda)}u_\lambda(x)$ and $g=2^{\ell(\mu)}v_\mu(x)$ and thus
$\langle
f,g
\rangle
=2^{\ell(\lambda)}\delta_{\lambda,\mu}
=\langle
\lambda|\mu
\rangle
$.
\end{proof}

\section{Neutral fermionic presentation of dual polynomials}\label{sec:dualbeta}

\subsection{Dual $\beta$-deformed neutral fermion $\dbphi_n$}
Let us introduce the ``dual'' $\beta$-deformed neutral fermion $\dbphi_n$ as
\begin{equation}\label{eq:dbphi}
\dbphi_n:=(-1)^n(\phi^{(-\beta)}_{-n})^\ast,
\end{equation}
which defines a $k$-linear map $\mathcal{F}\ni \ket{w}\mapsto \dbphi_n\ket{w}\in \mathcal{F}$.
Note that $\dbphi_n$ reduces to $\phi_n$ when $\beta=0$.

We can check that the formal series
$
\dbphi(z)=\sum_{n\in \ZZ}\dbphi_nz^n
$
defines a $k$-linear map $\mathcal{F}\to \mathcal{F}((z))$.
Let $(\dbphi)^\ast(z)=\sum_{n\in \ZZ}(\dbphi_n)^\ast z^{n}$ be the formal series obtained from $\dbphi(x)$ by the anti-involution \eqref{eq:anti2}.
Then, from \eqref{eq:dbphi}, we have
\begin{equation}\label{eq:dbphi_ast}
(\dbphi)^\ast(z)
=\phi^{(-\beta)}(-z^{-1})
=
\phi(-z^{-1}-\tfrac{\beta}{2})
\qquad \mbox{in}\quad \mathrm{Map}(\mathcal{E},\mathcal{E}((z))).
\end{equation}
Moreover, by using $\{\phi(z)\}^\ast=\phi(-z^{-1})$, we have
\begin{equation}
\dbphi(z)
=
\phi(\tfrac{z}{1+\frac{\beta}{2}z} )\qquad \mbox{in}\quad \mathrm{Map}(\mathcal{F},\mathcal{F}((z ))).
\end{equation}
\begin{proposition}\label{prop:quasi_duality}
For $m,n\geq 0$, we have
\[
[(\dbphi_m)^\ast,\bphi_n ]_+=
\begin{cases}
2, & m=n,\\
\beta,& m=n+1,\\
0, & \mbox{otherwise}.
\end{cases}
\]
\end{proposition}
\begin{proof}
Let $A(u)=\sum_{m=0}^\infty (\dbphi_m)^\ast u^m$ and $B(v)=\sum_{n=0}^\infty \bphi_n v^n$.
From \eqref{eq:beta_generating_function} and \eqref{eq:dbphi_ast}, we have 
\begin{align*}
[A(u),B(v)]_+
&=\sum_{m,n\geq 0}[\phi_{-m},\phi_n]_+\left(-u^{-1}-\frac{\beta}{2}\right)^{-m}\left(v+\frac{\beta}{2}\right)^n\\
&=\sum_{n=0}^{\infty}2(-1)^n\left(-u^{-1}-\frac{\beta}{2}\right)^{-n}\left(v+\frac{\beta}{2}\right)^n\\
&=\frac{2+\beta u}{1-uv}
=(2+\beta u)(1+uv+u^2v^2+\cdots)
\end{align*}
in the vector space $\mathrm{Hom}_k(\mathcal{E},\mathcal{E}[[u]][[v]])$, where $(-u^{-1}-\frac{\beta}{2})^{-1}$ is expanded as
\[
\left(-u^{-1}-\frac{\beta}{2}\right)^{-1}
=-\frac{u}{1+\frac{\beta}{2}u}
=-u+\frac{\beta}{2}u^2-\frac{\beta^2}{4}u^3+
\frac{\beta^3}{8}u^4-
\cdots.
\]
Comparing the coefficients of $u^mv^n$ on the both sides, we obtain the proposition.
\end{proof}

By using Corollary \ref{cor:e_Theta_bphi} and Proposition \ref{prop:quasi_duality}, we have
\begin{align*}
&[(\dbphi_m)^\ast ,e^{-\Theta}\bphi_ne^\Theta]_+
=\sum_{k=0}^\infty[(\dbphi_m)^\ast,(-\beta)^k\bphi_{n+k}]_+=
\begin{cases}
0, & m<n,\\
2, & m=n,\\
(-\beta)^{m-n}, & m>n.
\end{cases}
\end{align*}
(For the case when $m>n$, this equation is shown as below:
\begin{align*}
\sum_{k=0}^\infty[(\dbphi_m)^\ast,(-\beta)^k\bphi_{n+k}]_+
&=(-\beta)^{m-n-1} [(\dbphi_m)^\ast,\bphi_{m-1}]_++
(-\beta)^{m-n} [(\dbphi_m)^\ast,\bphi_{m}]_+\\
&=(-\beta)^{m-n-1}\cdot \beta +(-\beta)^{m-n}\cdot 2=(-\beta)^{m-n}.)
\end{align*}
Seeing this, we define
\[
I(m,n):=
\begin{cases}
0, & m<n,\\
1, & m=n=0,\\
2, & m=n>0,\\
(-\beta)^{m-n}, & m>n.
\end{cases}
\]
As is seen, $I(m,n)$ equals to $[(\dbphi_m)^\ast ,e^{-\Theta}\bphi_ne^\Theta]_+$ except for the case $m=n=0$.

The following proposition and lemma can be proved by similar arguments as those used to prove Proposition \ref{prop:eTheta_action_on_phi(beta)}.
\begin{proposition}\label{prop:eTheta_vs_dbphi}
We have
\[
e^{\Theta}(\dbphi)^\ast(z)e^{-\Theta}=(1+\beta z)^{-1}(\dbphi)^\ast(z)
\qquad \mbox{in}\quad \mathrm{Map}(\mathcal{E},\mathcal{E}((z))),
\]
which implies 
\[
e^{\Theta}(\dbphi_n)^\ast e^{-\Theta}
=
(\dbphi_n)^\ast-\beta (\dbphi_{n-1})^\ast
+\beta^2(\dbphi_{n-2})^\ast
-\cdots.
\]
\end{proposition}

\begin{lemma}
We have
\[
[\dbb_n,\dbphi(z)]=\frac{z^n-\overline{z}^n}{2}\dbphi(z)
\]
and thus
\[
e^{\dHb}\dbphi(z)=\exp\left(
\sum_{n=1}^\infty\frac{p_n}{n}(z^n-\overline{z}^n)
\right)\dbphi(z)e^{\dHb}.
\]
\end{lemma}

\subsection{Calculation of the expectation value ${}^g\langle \mu\,\vert\,\lambda \rangle^G$}

For a string of non-negative integers $n_1,\dots,n_r$, we write 
\begin{align*}
&{}^g\bra{n_1,\dots,n_r}:
=
\bra{0}
(\dbphi_{n_r})^\ast e^{-\Theta}
\cdots e^{-\Theta}
(\dbphi_{n_2})^\ast e^{-\Theta}
(\dbphi_{n_1})^\ast e^{-\Theta},\\
&\ket{n_1,\dots,n_r}^G:
=
\bphi_{n_1}e^{\Theta}
\bphi_{n_2}e^{\Theta}
\cdots
\bphi_{n_r}e^{\Theta}\ket{0}.
\end{align*}
In this section, we calculate the vacuum expectation value
\begin{equation}\label{eq:inner}
{}^g\langle \mu_1,\mu_2,\dots,\mu_r|\lambda_1,\lambda_2,\dots,\lambda_s\rangle^G
\end{equation}
for strict partitions $\lambda=(\lambda_1>\dots>\lambda_s>0)$ and $\mu=(\mu_1>\dots>\mu_r>0)$.
This gives a $\beta$-deformation of Lemma \ref{lemma:basic_duality}.

Simplest cases are given as follows:
\begin{lemma}\label{eq:basic}
We have
\begin{enumerate}
\def\labelenumi{(\roman{enumi})}
\item\label{eq:basic_1} ${}^g\langle \emptyset |\emptyset \rangle^G=1$,
\item\label{eq:basic_2} ${}^g\langle 0 | 0 \rangle^G=1$.
\end{enumerate}
\end{lemma}
\begin{proof}
Eq \eqref{eq:basic_1} is obvious.
Eq \eqref{eq:basic_2} is given as follows:
\begin{align*}
{}^g\langle 0 | 0 \rangle^G&=
\bra{0}(\dbphi_0)^\ast e^{-\Theta}\bphi_0e^\Theta \ket{0}\\
&=
\bra{0} (\dbphi_0)^\ast(\bphi_0-\beta \bphi_1+\cdots)\ket{0}
\hspace{3em}(\mbox{By Corollary \ref{cor:e_Theta_bphi}})
\\
&=\bra{0}(\phi_0^2+\cdots)\ket{0}=1.
\end{align*}
\end{proof}

To calculate the general value \eqref{eq:inner}, we need the following three technical Lemmas \ref{lemma:tech1}--\ref{lemma:tech3}.

\begin{lemma}\label{lemma:tech1}
If $N>\mu_1>\mu_2>\dots>\mu_r\geq 0$, then ${}^g\bra{\mu_1,\mu_2,\dots,\mu_r}\bphi_N=0$.
\end{lemma}
\begin{proof}
By Proposition \ref{prop:eTheta_vs_dbphi} and $\bra{0}e^\Theta=\bra{0}$, the vector ${}^g\bra{\mu_1,\dots,\mu_r}$ should be a linear combination of vectors of the form
\[
\bra{0}(\dbphi_{n_r})^\ast\cdots (\dbphi_{n_2})^\ast(\dbphi_{n_1})^\ast,\quad
(N>n_1>n_2>\dots>n_r\geq 0).
\]
Noting $\bra{0}\bphi_N=0$ and $[(\dbphi_{n_r})^\ast\cdots (\dbphi_{n_2})^\ast(\dbphi_{n_1})^\ast,\bphi_N]_+=0$, which is shown from Proposition \ref{prop:quasi_duality}, we obtain ${}^g\bra{\mu_1,\mu_2,\dots,\mu_r}\bphi_N=0$.
\end{proof}

\begin{lemma}\label{lemma:tech2}
If $M>\lambda_1>\dots>\lambda_r\geq 0$, then $(\dbphi_{M+1})^\ast\ket{\lambda_1,\dots,\lambda_r}^G=0$.
\end{lemma}
\begin{proof}
By Proposition \ref{prop:eTheta_action_on_phi(beta)}, the vector $\ket{\lambda_1,\dots,\lambda_r}^G$ is a linear combination of vectors of the form
\[
\bphi_{n_1}\bphi_{n_2}\dots \bphi_{n_r}e^{r\Theta}\ket{0},\quad
(M>n_1\geq n_2\geq \dots \geq n_r\geq 0).
\]
Noting $M\geq r$ and $(\dbphi_n)^\ast\ket{0}=0$ for $n>0$, we have
\begin{align*}
&(\dbphi_{M+1})^\ast \bphi_{n_1}\bphi_{n_2}\dots \bphi_{n_r}e^{r\Theta}\ket{0}\\
&=
(-1)^r \bphi_{n_1}\bphi_{n_2}\dots \bphi_{n_r}(\dbphi_{M+1})^\ast e^{r\Theta}\ket{0}\\
&=
(-1)^r \bphi_{n_1}\bphi_{n_2}\dots \bphi_{n_r}
e^{r\Theta}((\dbphi_{M+1})^\ast +r\beta(\dbphi_{M})^\ast+\cdots +\beta ^r(\dbphi_{M-r+1})^\ast)\ket{0}\\
&\hspace{23em}
(\mbox{By Proposition \ref{prop:eTheta_vs_dbphi}})
\\
&=0.
\end{align*}
\end{proof}

\begin{lemma}\label{lemma:tech3}
If $\mu_1>\dots>\mu_r\geq  0$ and $\lambda_1>\dots>\lambda_s\geq 0$, we have
\begin{equation}\label{eq:desired}
{}^g\langle
\mu_1,\mu_2,\dots,\mu_r|
\lambda_1,\lambda_2,\dots,\lambda_s
\rangle^G
=I(\mu_1,\lambda_1)\cdot
{}^g\langle
\mu_2,\dots,\mu_r|
\lambda_2,\dots,\lambda_s
\rangle^G.
\end{equation}
\end{lemma}
\begin{proof}
If $\mu_1=0$ and $\lambda_1=0$ (and $r=s=1$ automatically), it follows that
${}^g\langle
0
|
0
\rangle^G
=I(0,0)\cdot
{}^g\langle
\emptyset|
\emptyset
\rangle^G
$
from Lemma \ref{eq:basic}.
Thus \eqref{eq:desired} is true in this case.

Suppose $\mu_1>0$ or $\lambda_1>0$.
In this case, we have
\begin{align*}
&{}^g\langle
\mu_1,\mu_2,\dots,\mu_r|
\lambda_1,\lambda_2,\dots,\lambda_s
\rangle^G\\
&=
{}^g\bra{\mu_2,\dots,\mu_r}(\dbphi_{\mu_1})^\ast e^{-\Theta}\bphi_{\lambda_1}e^{\Theta}
\ket{\lambda_2,\dots,\lambda_s}^G\\
&=
{}^g\bra{\mu_2,\dots,\mu_r}
\left(
[(\dbphi_{\mu_1})^\ast, e^{-\Theta}\bphi_{\lambda_1}e^{\Theta}]_+
-
e^{-\Theta}\bphi_{\lambda_1}e^{\Theta}(\dbphi_{\mu_1})^\ast
\right)
\ket{\lambda_2,\dots,\lambda_s}^G\\
&=
I(\mu_1,\lambda_1)\cdot
{}^g\langle
\mu_2,\dots,\mu_r|
\lambda_2,\dots,\lambda_s
\rangle^G\\
&\hspace{12em}-
{}^g\bra{\mu_2,\dots,\mu_r}
e^{-\Theta}\bphi_{\lambda_1}e^{\Theta}(\dbphi_{\mu_1})^\ast
\ket{\lambda_2,\dots,\lambda_s}^G.
\end{align*}
Let $T={}^g\bra{\mu_2,\dots,\mu_r}
e^{-\Theta}\bphi_{\lambda_1}e^{\Theta}(\dbphi_{\mu_1})^\ast
\ket{\lambda_2,\dots,\lambda_s}^G$ be the second term.
If $\lambda_1\geq \mu_1$ (and automatically $\lambda_1>0$), it follows that ${}^g\bra{\mu_2,\dots,\mu_r}
e^{-\Theta}\bphi_{\lambda_1}e^{\Theta}=0$ from Lemma \ref{lemma:tech1}.
This implies $T=0$.
If $\lambda_1< \mu_1$ (and $\mu_1>0$), we have $(\dbphi_{\mu_1})^\ast
\ket{\lambda_2,\dots,\lambda_s}^G=0$ by Lemma \ref{lemma:tech2}, which leads $T=0$.
In each case, \eqref{eq:desired} is true.
\end{proof}

By using the lemmas we have just proved, we conclude that the vacuum expectation value \eqref{eq:inner} is expressed as
\begin{equation}\label{eq:bilinar_form_explicit}
{}^g\langle
\mu_1,\mu_2,\dots,\mu_r|
\lambda_1,\lambda_2,\dots,\lambda_s
\rangle^G
=\begin{cases}
\prod_{i=1}^rI(\mu_i,\lambda_i), & r=s,\\
0, & \mbox{otherwise}.
\end{cases}
\end{equation}
Write $\ket{\lambda_1,\dots,\lambda_r}^g:=({}^g\bra{\lambda_1,\dots,\lambda_r})^\ast$.
We now define a new symmetric polynomial $o_\lambda(x)$ as
\begin{equation}\label{eq:neutral_o}
o_\lambda(x)=
\begin{cases}
2^{-r}\cdot \bra{0}e^{\dHb(x)}\ket{\lambda_1,\lambda_2,\dots,\lambda_r}^g, & r\mathrm{:even},\\
2^{-r}\cdot \bra{0}e^{\dHb(x)}\ket{\lambda_1,\lambda_2,\dots,\lambda_r,0}^g, & r\mathrm{:odd}.
\end{cases}
\end{equation}

\begin{theorem}\label{thm:quasi_innerproduct}
Let $\lambda$, $\mu$ be strict partitions and $\ell(\lambda)$, $\ell(\mu)$ be their lengths, respectively.
Write $\ell(\lambda)'$ $($\textrm{resp.} $\ell(\mu)')$ be the smallest integer that is equal to or greater than $\ell(\lambda)$ $($\textrm{resp.} $\ell(\mu))$.
For a skew shape $\mu/\lambda$, we let $\mathrm{row}(\mu/\lambda)$ denote the number of non-empty rows of $\mu/\lambda$.
Then we have
\[
\langle
GQ_\lambda,o_\mu
\rangle
=\begin{cases}
\dfrac{(-\beta)^{\zet{\mu}-\zet{\lambda}}}{2^{\mathrm{row}(\mu/\lambda)}}, & \mu\supset \lambda \mbox{ and }\ell(\mu)'=\ell(\lambda)',\\
0, & \mbox{otherwise}.
\end{cases}
\]
\end{theorem}
\begin{proof}
Noting the equation $\mathrm{row}(\mu/\lambda)=\ell(\mu)-\sharp\{i\,\vert\,\mu_i=\lambda_i>0\}$ for $\mu\supset \lambda$, we know that the theorem follows directly from Proposition \ref{prop:bilin_and_expect} and \eqref{eq:bilinar_form_explicit}.
\end{proof}

From Theorem \ref{thm:quasi_innerproduct}, we can find a unique family of symmetric polynomials $gp_\lambda(x)\in g\Gamma$ $(\lambda:\mbox{strict})$ which satisfies
\begin{equation}\label{eq:dual_polynomials}
\langle GQ_\lambda,gp_\mu \rangle=\delta_{\lambda,\mu}.
\end{equation}
In fact, by using M\"{o}bius inversion theorem, we know that there exists a unique family of rational numbers $k_{\lambda,\mu}\in \QQ$ such that the sum
\begin{equation}\label{eq:gp_to_o}
gp_\lambda(x)=o_\lambda(x)+
\sum_{\substack{\mu\mathrm{:strict},\\ \lambda\supsetneq \mu,\, \ell(\lambda)'=\ell(\mu)'}}
k_{\lambda,\mu}\cdot \beta^{\zet{\lambda}-\zet{\mu}} o_\mu(x)
\end{equation}
satisfies \eqref{eq:dual_polynomials}.
Theorem \ref{thm:quasi_innerproduct} is now rewritten as
\begin{equation}\label{eq:o_to_gp}
o_\lambda(x)=gp_\lambda(x)+\sum_{\substack{\mu\mathrm{:strict},\\ \lambda\supsetneq \mu,\, \ell(\lambda)'=\ell(\mu)'}} \dfrac{(-\beta)^{\zet{\lambda}-\zet{\mu}}}{2^{\mathrm{row}(\lambda/\mu)}}gp_\mu(x).
\end{equation}

\begin{example}
For a one-row partition $(n)$, we write $o_n(x)=o_{(n)}(x)$ and $gp_n(x)=gp_{(n)}(x)$.
From \eqref{eq:gp_to_o} and \eqref{eq:o_to_gp}, we have
\[
o_n(x)=gp_n(x)
-\frac{\beta}{2}gp_{n-1}(x)
+\frac{\beta^2}{2}gp_{n-2}(x)
-\frac{\beta^3}{2}gp_{n-3}(x)
+\cdots+
\frac{(-\beta)^{n-1}}{2}gp_1(x)
\]
and
\[
gp_n(x)=o_n(x)+\frac{\beta}{2}o_{n-1}(x)
-\frac{\beta^2}{4}o_{n-2}(x)+\cdots-\frac{(-\beta)^{n-1}}{2^{n-1}}o_1(x).
\]
\end{example}

\section{On Pfaffian formulas for $o_\lambda(x)$}\label{sec:Pfaffian_formula_dual}

Let us proceed to Pfaffian formulas for dual functions.
Unfortunately, we do not have a Pfaffian-type formula for $gp_\lambda(x)$ itself at this stage.
Instead, we show two types of Pfaffian formulas for $o_\lambda(x)$, which are proved by straightforward calculation similar to that used in Section \ref{sec:GQ_lambda}.

\subsection{Pfaffian formula I for $o_\lambda$}

Write
\[
\theta:=\Theta^\ast=
2\left(\frac{\beta}{2} b_{1}+\frac{\beta^3}{2^3}\frac{b_{3}}{3}+
\frac{\beta^5}{2^5}\frac{b_{5}}{5}+\cdots\right).
\]
Then, from \eqref{eq:neutral_o}, we have
\[
o_\lambda(x)=
\begin{cases}
2^{-r}\cdot \bra{0}e^{\dHb(x)}e^{-\theta}\dbphi_{\lambda_1}e^{-\theta}\cdots \dbphi_{\lambda_r}\ket{0},
& \mbox{if $r$ is even},\\
2^{-r}\cdot \bra{0}e^{\dHb(x)}e^{-\theta}\dbphi_{\lambda_1}e^{-\theta}\cdots \dbphi_{\lambda_r}e^{-\theta}\dbphi_0\ket{0},
& \mbox{if $r$ is odd}.
\end{cases}
\]

Let us define a new symmetric polynomial $q^{[\beta]}_n(x)$ by
\[
\sum_{n=0}^\infty q^{[\beta]}_n(x)z^n=
\bra{0}e^{\dHb}\dbphi(z)\dbphi_0\ket{0}=\exp\left(
\sum_{n=1}^{\infty}\frac{p_n}{n}(z^n-\overline{z}^n)
\right)
=\prod_{i}\frac{1-x_i\overline{z}}{1-x_iz}.
\]
If we put $\beta=0$, $q^{[\beta]}_n(x)$ comes back to $q_n(x)$.
Noting $e^{-\theta}\ket{0}=\ket{0}$, $e^{-\theta}\dbphi_0\ket{0}=\dbphi_0\ket{0}$, and the relation 
\[
e^{-\theta}\dbphi(z)e^\theta=(1+\beta z)^{-1}\dbphi(z)
\] 
that is given from Proposition \ref{prop:eTheta_vs_dbphi}, we have
\begin{align*}
\sum_{n=0}^{\infty}o_{n}(x)z^n
&=
2^{-1}\bra{0}e^{\dHb}e^{-\theta}\dbphi(z)e^{-\theta}\dbphi_0\ket{0}\\
&=
2^{-1}(1+\beta z)^{-1}\bra{0}e^{\dHb}\dbphi(z)\dbphi_0\ket{0}\\
&=
\frac{1}{2(1+\beta z)}\prod_{i}\frac{1-x_i\overline{z}}{1-x_iz}.
\end{align*}
This equation can be seen as a counterpart of \eqref{eq:GQ_n}.
From this, we obtain 
\begin{equation}\label{eq:special_case}
o_{n}(x)=\frac{1}{2}\left(q^{[\beta]}_n(x)-\beta q^{[\beta]}_{n-1}(x)+\beta^2 q^{[\beta]}_{n-2}(x)
-\cdots+(-\beta)^nq^{[\beta]}_0(x)\right).
\end{equation}

Let us consider the formal function $\Psi=\Psi(z_1,\dots,z_r)$ defined as
\[
\Psi=
\begin{cases}
\bra{0}
e^{\dHb}e^{-\theta}
\dbphi(z_1)e^{-\theta}
\dbphi(z_2)e^{-\theta}
\cdots
\dbphi(z_r)
\ket{0}, & \mbox{if $r$ is even},\\
\bra{0}
e^{\dHb}e^{-\theta}
\dbphi(z_1)e^{-\theta}
\dbphi(z_2)e^{-\theta}
\cdots
\dbphi(z_r)e^{-\theta}
\dbphi_0
\ket{0}, & \mbox{if $r$ is odd},
\end{cases}
\]
in which the coefficient of $z_1^{\lambda_1}\dots z_r^{\lambda_r}$ is $o_\lambda(x)$.
By putting 
\begin{align*}
&B_i=e^{-i\theta}\dbphi(z_i)e^{i\theta}=(1+\beta z_i)^{-i}\dbphi(z_i),\\ &B_{r+1}=e^{-(r+1)\theta}\dbphi_0e^{(r+1)\theta}.
\end{align*}
and noting $e^\theta\ket{0}=\ket{0}$, we obtain
\begin{align*}
\Psi
&=
\bra{0}
e^{\dHb}
B_1B_2\cdots B_{r'}
\ket{0}
=
\Pf\left(
\bra{0}
e^{\dHb}B_iB_j
\ket{0}
\right)_{1\leq i<j\leq r'}.
\end{align*}
Therefore, we have
\begin{equation}\label{eq:Psi_Pfaffian}
\Psi=\Pf\left(\mathcal{Z}_{i,j}\right)_{1\leq i<j\leq r'},
\end{equation}
where
\[
\mathcal{Z}_{i,j}=
\begin{cases}
(1+\beta z_i)^{-i}
(1+\beta z_j)^{-j}
q^{[\beta]}(z_i)
q^{[\beta]}(z_j)
\dfrac{z_i-z_j}{z_i+z_j+\beta z_iz_j}, & j\neq r+1,\\
(1+\beta z_i)^{-i}q^{[\beta]}(z_i),
& j=r+1.
\end{cases}
\]
Here $\dfrac{z_i-z_j}{z_i+z_j+\beta z_iz_j}$ is expanded as
\[
\dfrac{1-z_jz_i^{-1}}{1+z_jz_i^{-1}+\beta z_j}
=1-(2z_i^{-1}+\beta)z_j+(2z_i^{-2}+3\beta z_i^{-1}+\beta^2)z_j^{2}-\cdots.
\] 

Let $g^{i,j}_{p,q}$ be the element of $\QQ[\beta]$ defined by
\begin{gather*}
\begin{cases}
\displaystyle
(1+\beta z)^{-i}(1+\beta w)^{-j}\frac{z-w}{z+w+\beta zw}=
\sum_{p+q\geq 0,\, q\geq 0}g^{i,j}_{p,q}z^pw^q,&
j\neq r+1,
\\
\displaystyle
(1+\beta z)^{-i}=\sum_{p\geq 0}g^{i,r+1}_{p}z^p,
&
j=r+1.
\end{cases}
\end{gather*}
Then, by comparing coefficients on the both sides of (\ref{eq:Psi_Pfaffian}), we have the explicit formula:
\begin{proposition}
We have
\[
o_\lambda(x)=2^{-r}\Pf
\left(
\zeta_{i,j}
\right)_{1\leq i<j\leq r'},
\]
where 
\[
\zeta_{i,j}=
\begin{cases}
\sum_{p+q\geq 0,\, q\geq 0}g^{i,j}_{p,q}q^{[\beta]}_{\lambda_i-p}(x)q^{[\beta]}_{\lambda_j-q}(x), & j\neq r+1,\\
\sum_{p\geq 0}g^{i,r+1}_{p}q^{[\beta]}_{\lambda_i-p}(x), & j=r+1.
\end{cases}
\]
\end{proposition}
Note that \eqref{eq:special_case} is a special case of this equation.

\subsection{Pfaffian formula II for $o_\lambda$}
Similarly to \S \ref{sec:anotherP}, we can derive another Pfaffian formula for $o_\lambda(x)$.
Since
\begin{align*}
&\bra{0}
e^{\dHb}B_iB_j
\ket{0}\\
&=
(1+\beta z_i)^{-(i-1)}(1+\beta z_j)^{-(j-2)}
\bra{0}
e^{\dHb}e^{-\theta}\dbphi(z_i)e^{-\theta}\dbphi(z_j)
\ket{0}\\
&=
(1+\beta z_i)^{-(i-1)}(1+\beta z_j)^{-(j-2)}
\Psi(z_i,z_j),
\end{align*}
we have
\[
\Psi(z_1,\dots,z_n)=\Pf(\mathcal{K}_{i,j})_{1\leq i<j\leq r'},
\]
where
\[
\mathcal{K}_{i,j}=
\begin{cases}
(1+\beta z_i)^{-(i-1)}
(1+\beta z_j)^{-(j-2)}
\Psi(z_i,z_j),
& j\neq r+1,\\
(1+\beta z_i)^{-(i-1)}
\Psi(z_i),
& j=r+1.
\end{cases}
\]
By comparing coefficients on the both sides, we obtain
\[
o_\lambda(x)=\Pf(\kappa_{i,j})_{1\leq i<j\leq r'},
\]
where
\[
\kappa_{i,j}=
\begin{cases}
\sum_{k=0}^{\infty}\sum_{l=0}^{\infty}\binom{1-i}{k}\binom{2-j}{l}o_{(\lambda_i-k,\lambda_j-l)}(x), & j\neq r+1,\\
\sum_{k=0}^{\infty}\binom{1-i}{k}o_{\lambda_i-k}(x), & j=r+1.
\end{cases}
\]


\section*{Acknowledgments}

This work is partially supported by JSPS Kakenhi Grant Number 19K03605.
The author is very grateful to Prof.~Takeshi Ikeda for his valuable advice.
The author also would like to thank Prof.~Hiroshi Naruse for important suggestions on the manuscript.


\begin{thebibliography}{10}

\bibitem{baker1995symmetric}
TH~Baker, \emph{Symmetric function products and plethysms and the boson-fermion
  correspondence}, Journal of Physics A: Mathematical and General \textbf{28}
  (1995), no.~3, 589.

\bibitem{buch2002littlewood}
Anders~S Buch, \emph{A {L}ittlewood-{R}ichardson rule for the {$K$}-theory of
  {G}rassmannians}, Acta mathematica \textbf{189} (2002), no.~1, 37--78.

\bibitem{date1983method}
Etsuro Date, Michio Jimbo, and Tetsuji Miwa, \emph{Method for generating
  discrete soliton equations. {V}}, Journal of the Physical Society of Japan
  \textbf{52} (1983), no.~3, 766--771.

\bibitem{fulton2006schubert}
William Fulton and Piotr Pragacz, \emph{Schubert varieties and degeneracy
  loci}, Springer, 2006.

\bibitem{gorbounov2017quantum}
Vassily Gorbounov and Christian Korff, \emph{Quantum integrability and
  generalised quantum {S}chubert calculus}, Advances in Mathematics
  \textbf{313} (2017), 282--356.

\bibitem{graham2015excited}
William Graham and Victor Kreiman, \emph{Excited {Y}oung diagrams, equivariant
  {$K$}-theory, and {S}chubert varieties}, Transactions of the American
  Mathematical Society \textbf{367} (2015), no.~9, 6597--6645.

\bibitem{hudson2016pfaffian}
Thomas Hudson, Takeshi Ikeda, Tomoo Matsumura, and Hiroshi Naruse,
  \emph{Pfaffian formula for {$K$}-theory of odd orthogonal {G}rassmannians},
  arXiv preprint:1602.04448 (2016).

\bibitem{HUDSON2017115}
\bysame, \emph{Degeneracy loci classes in {$K$}-theory — determinantal and
  {P}faffian formula}, Advances in Mathematics \textbf{320} (2017), 115--156.

\bibitem{IKEDA201322}
Takeshi Ikeda and Hiroshi Naruse, \emph{{$K$}-theoretic analogues of factorial
  {Schur} {$P$}- and {$Q$}-functions}, Advances in Mathematics \textbf{243}
  (2013), 22--66.

\bibitem{ikeda2011bumping}
Takeshi Ikeda, Hiroshi Naruse, and Yasuhide Numata, \emph{Bumping algorithm for
  set-valued shifted tableaux}, Discrete Mathematics and Theoretical Computer
  Science, Discrete Mathematics and Theoretical Computer Science, 2011,
  pp.~527--538.

\bibitem{iwao2020freefermions}
Shinsuke Iwao, \emph{Free-fermions and skew stable {G}rothendieck polynomials},
  arXiv preprint:2004.09499 (2020).

\bibitem{iwao2020freefermion}
\bysame, \emph{Grothendieck polynomials and~the~boson-fermion correspondence},
  Algebraic Combinatorics \textbf{3} (2020), no.~5, 1023--1040.

\bibitem{jimbo1983solitons}
Michio Jimbo and Tetsuji Miwa, \emph{Solitons and infinite dimensional {L}ie
  algebras}, Publications of the Research Institute for Mathematical Sciences
  \textbf{19} (1983), no.~3, 943--1001.

\bibitem{jozefiak1991schur}
Tadeusz J{\'o}zefiak, \emph{Schur {Q}-functions and cohomology of isotropic
  {G}rassmannians}, Mathematical Proceedings of the Cambridge Philosophical
  Society, vol. 109, Cambridge University Press, 1991, pp.~471--478.

\bibitem{kirillov2017construction}
Anatol~N Kirillov and Hiroshi Naruse, \emph{Construction of double
  {Grothendieck} polynomials of classical types using {IdCoxeter} algebras},
  Tokyo Journal of Mathematics \textbf{39} (2017), no.~3, 695--728.

\bibitem{kostant1990t}
Bertram Kostant, Shrawan Kumar, et~al., \emph{{$T$}-equivariant {$K$}-theory of
  generalized flag varieties}, Journal of Differential Geometry \textbf{32}
  (1990), no.~2, 549--603.

\bibitem{macdonald1998symmetric}
Ian~G Macdonald, \emph{Symmetric functions and {Hall} polynomials}, Oxford
  university press, 1998.

\bibitem{motegi2020integrability}
Kohei Motegi, \emph{Integrability approach to
  {F}eh{\'e}r-{N}{\'e}methi-{R}im{\'a}nyi-{G}uo-{S}un type identities for
  factorial {G}rothendieck polynomials}, Nuclear Physics B (2020), 114998.

\bibitem{motegi2013vertex}
Kohei Motegi and Kazumitsu Sakai, \emph{Vertex models, {TASEP} and
  {Grothendieck} polynomials}, Journal of Physics A: Mathematical and
  Theoretical \textbf{46} (2013), no.~35, 355201.

\bibitem{motegi2014k}
\bysame, \emph{{$K$}-theoretic boson--fermion correspondence and melting
  crystals}, Journal of Physics A: Mathematical and Theoretical \textbf{47}
  (2014), no.~44, 445202.

\bibitem{nakagawa2017generating}
Masaki Nakagawa and Hiroshi Naruse, \emph{Generating functions for the
  universal {Hall}-{Littlewood} {$P$}- and {$Q$}-functions}, arXiv
  preprint:1705.04791 (2017).

\bibitem{nakagawa2018universal}
\bysame, \emph{Universal {G}ysin formulas for the universal {H}all-{L}ittlewood
  functions}, Contemporary Mathematics, vol. 708, American Mathematical
  Society, 2018, pp.~201--244.

\bibitem{nakagawa2016generalized}
Masaki Nakagawa, Hiroshi Naruse, et~al., \emph{Generalized (co) homology of the
  loop spaces of classical groups and the universal factorial schur {$ P $}-and
  {$ Q $}-functions}, Schubert Calculus—Osaka 2012, Mathematical Society of
  Japan, 2016, pp.~337--417.

\bibitem{pragacz1988enumerative}
Piotr Pragacz, \emph{Enumerative geometry of degeneracy loci}, Annales
  scientifiques de l'{\'E}cole Normale Sup{\'e}rieure, vol.~21, 1988,
  pp.~413--454.

\bibitem{Pragacz1991sqpoly}
\bysame, \emph{Algebro - {G}eometric applications of {S}chur {$S$}- and
  {$Q$}-polynomials}, Topics in Invariant Theory, Springer Berlin Heidelberg,
  1991, pp.~130--191.

\bibitem{Schur1911}
Issai Schur, \emph{{\"{U}ber} die {Darstellung} der symmetrischen und der
  alternierenden {Gruppe} durch gebrochene lineare {Substitutionen}.}, Journal
  f\"{u}r die reine und angewandte Mathematik \textbf{139} (1911), 155--250.

\bibitem{STEMBRIDGE198987}
John~R Stembridge, \emph{Shifted tableaux and the projective representations of
  symmetric groups}, Advances in Mathematics \textbf{74} (1989), no.~1,
  87--134.

\end{thebibliography}
\providecommand{\bysame}{\leavevmode\hbox to3em{\hrulefill}\thinspace}
\providecommand{\MR}{\relax\ifhmode\unskip\space\fi MR }
\providecommand{\MRhref}[2]{%
  \href{http://www.ams.org/mathscinet-getitem?mr=#1}{#2}
}
\providecommand{\href}[2]{#2}


\end{document}